\documentclass[10pt,a4paper]{amsart}
\usepackage[margin=20mm]{geometry}

\usepackage{amssymb,amsmath}
\usepackage{amsthm}
\usepackage{bbm}
\usepackage{stmaryrd}
\usepackage{hyperref}
\usepackage[usenames,dvipsnames]{xcolor}
\usepackage{color}

\input xy
\xyoption{all} 

\numberwithin{equation}{section}

\newtheorem{theorem}{Theorem}[section]
\newtheorem{corollary}[theorem]{Corollary}
\newtheorem{lemma}[theorem]{Lemma}
\newtheorem{proposition}[theorem]{Proposition}

\theoremstyle{definition}
\newtheorem{definition}[theorem]{Definition}
\newtheorem{remark}[theorem]{Remark}
\newtheorem{example}[theorem]{Example}

\newcommand{\Dleft}{[\hspace{-1.5pt}[}
\newcommand{\Dright}{]\hspace{-1.5pt}]}
\newcommand{\SN}[1]{\Dleft #1 \Dright}

\newcommand{\Id}{\mathbbmss{1}}

\newcommand{\rmd}{\textnormal{d}}

\newcommand{\rmh}{\textnormal{h}}

\newcommand{\rml}{\textnormal{l}}

\newcommand{\rmp}{\textnormal{p}}
\newcommand{\N}{\textnormal{N}}

\DeclareMathOperator{\w}{w}

\font\black=cmbx10 \font\sblack=cmbx7 \font\ssblack=cmbx5 \font\blackital=cmmib10  \skewchar\blackital='177
\font\sblackital=cmmib7 \skewchar\sblackital='177 \font\ssblackital=cmmib5 \skewchar\ssblackital='177
\font\sanss=cmss10 \font\ssanss=cmss8 
\font\sssanss=cmss8 scaled 600 \font\blackboard=msbm10 \font\sblackboard=msbm7 \font\ssblackboard=msbm5
\font\caligr=eusm10 \font\scaligr=eusm7 \font\sscaligr=eusm5  \font\fraktur=eufm10
\font\sfraktur=eufm7 \font\ssfraktur=eufm5 
\font\bsymb=cmsy10 scaled\magstep2
\def\all#1{\setbox0=\hbox{\lower1.5pt\hbox{\bsymb
       \char"38}}\setbox1=\hbox{$_{#1}$} \box0\lower2pt\box1\;}
\def\exi#1{\setbox0=\hbox{\lower1.5pt\hbox{\bsymb \char"39}}
       \setbox1=\hbox{$_{#1}$} \box0\lower2pt\box1\;}

\def\tx#1{{\fam0\relax#1}}

\newfam\bifam
\textfont\bifam=\blackital \scriptfont\bifam=\sblackital \scriptscriptfont\bifam=\ssblackital

\newfam\blfam
\textfont\blfam=\black \scriptfont\blfam=\sblack \scriptscriptfont\blfam=\ssblack

\newfam\bbfam
\textfont\bbfam=\blackboard \scriptfont\bbfam=\sblackboard \scriptscriptfont\bbfam=\ssblackboard

\newfam\ssfam
\textfont\ssfam=\sanss \scriptfont\ssfam=\ssanss \scriptscriptfont\ssfam=\sssanss
\def\sss#1{{\fam\ssfam\relax#1}}

\newfam\clfam
\textfont\clfam=\caligr \scriptfont\clfam=\scaligr \scriptscriptfont\clfam=\sscaligr

\newfam\frfam
\textfont\frfam=\fraktur \scriptfont\frfam=\sfraktur \scriptscriptfont\frfam=\ssfraktur

\def\hpb#1{\setbox0=\hbox{${#1}$}
    \copy0 \kern-\wd0 \kern.2pt \box0}
\def\vpb#1{\setbox0=\hbox{${#1}$}
    \copy0 \kern-\wd0 \raise.08pt \box0}

\def\pmb#1{\setbox0\hbox{${#1}$} \copy0 \kern-\wd0 \kern.2pt \box0}
\def\pmbb#1{\setbox0\hbox{${#1}$} \copy0 \kern-\wd0
      \kern.2pt \copy0 \kern-\wd0 \kern.2pt \box0}
\def\pmbbb#1{\setbox0\hbox{${#1}$} \copy0 \kern-\wd0
      \kern.2pt \copy0 \kern-\wd0 \kern.2pt
    \copy0 \kern-\wd0 \kern.2pt \box0}
\def\pmxb#1{\setbox0\hbox{${#1}$} \copy0 \kern-\wd0
      \kern.2pt \copy0 \kern-\wd0 \kern.2pt
      \copy0 \kern-\wd0 \kern.2pt \copy0 \kern-\wd0 \kern.2pt \box0}
\def\pmxbb#1{\setbox0\hbox{${#1}$} \copy0 \kern-\wd0 \kern.2pt
      \copy0 \kern-\wd0 \kern.2pt
      \copy0 \kern-\wd0 \kern.2pt \copy0 \kern-\wd0 \kern.2pt
      \copy0 \kern-\wd0 \kern.2pt \box0}

\mathchardef\za="710B  
\mathchardef\zb="710C  
\mathchardef\zg="710D  
\mathchardef\zd="710E  
\mathchardef\zve="710F 
\mathchardef\zz="7110  
\mathchardef\zh="7111  
\mathchardef\zvy="7112 
\mathchardef\zi="7113  
\mathchardef\zk="7114  
\mathchardef\zl="7115  
\mathchardef\zm="7116  
\mathchardef\zn="7117  
\mathchardef\zx="7118  
\mathchardef\zp="7119  
\mathchardef\zr="711A  
\mathchardef\zs="711B  
\mathchardef\zt="711C  
\mathchardef\zu="711D  
\mathchardef\zvf="711E 
\mathchardef\zq="711F  
\mathchardef\zc="7120  
\mathchardef\zw="7121  
\mathchardef\ze="7122  
\mathchardef\zy="7123  
\mathchardef\zf="7124  
\mathchardef\zvr="7125 
\mathchardef\zvs="7126 
\mathchardef\zf="7127  
\mathchardef\zG="7000  
\mathchardef\zD="7001  
\mathchardef\zY="7002  
\mathchardef\zL="7003  
\mathchardef\zX="7004  
\mathchardef\zP="7005  
\mathchardef\zS="7006  
\mathchardef\zU="7007  
\mathchardef\zF="7008  
\mathchardef\zW="700A  
\mathchardef\zC="7009  

\newcommand{\be}{\begin{equation}}
\newcommand{\ee}{\end{equation}}

\newcommand{\bea}{\begin{eqnarray}}
\newcommand{\eea}{\end{eqnarray}}
\def\*{{\textstyle *}}
\newcommand{\R}{{\mathbb R}}

\newcommand{\s}{{\textstyle *}}

\newcommand{\ti}{\times}

\def\Sec{\sss{Sec}}
\def\Vect{\sss{Vect}}

\def\sT{{\sss T}}

\def\sV{{\sss V}}

\def\xi{\tx{i}}

\def\s*{{\scriptstyle *}}

\def\cO{\mathcal{O}}

\newcommand{\beas}{\begin{eqnarray*}}
\newcommand{\eeas}{\end{eqnarray*}}

\def\half{\frac{1}{2}}

\newcommand{\dd}{\mathbf{d}}

\begin{document}
\bibliographystyle{plain}

\author{Andrew James Bruce}
        \address{Mathematics Research Unit, University of Luxembourg, Maison du Nombre 6, avenue de la Fonte,
L-4364 Esch-sur-Alzette, Luxembourg.}\email{andrewjamesbruce@googlemail.com}

\date{\today}
\title{Connections adapted to non-negatively graded structures}

\begin{abstract}
Graded bundles are a particularly nice class of graded manifolds and represent a natural generalisation  of vector bundles. By exploiting the formalism of supermanifolds to describe Lie algebroids we define the notion of a \emph{weighted $A$-connection} on a graded bundle. In a natural sense weighted $A$-connections are adapted to the basic geometric structure of a graded bundle in the same way as  linear $A$-connections are adapted to the structure of a vector bundle. This notion generalises directly to multi-graded bundles and in particular we present  the notion of a \emph{bi-weighted $A$-connection} on a double vector bundle. We prove the existence of such adapted connections and use them to define (quasi-)actions of Lie algebroids on graded bundles.
\end{abstract}
\maketitle
\begin{small}
\noindent \textbf{MSC (2010)}: 16W25:~53B05;~53B15;~53D17;~58A50.\smallskip

\noindent \textbf{Keywords}: graded bundles;~ double vector bundles;~Lie algebroids;~connections;~actions.
\end{small}

\tableofcontents

\section{Introduction}\label{sec:Intro}
The notion of a \emph{connection} in many different guises, such as a covariant derivative  or a horizontal distribution,   can be found throughout modern differential geometry.  Of central importance are linear connections on vector bundles, of which affine connections constitute a prime example. Connections often lead to geometric invariants and convenient formulations thereof. A prominent example of the r\v{o}le of connections in constructing invariants is to be found in  Chern--Weil theory of the characteristic classes of principal bundles. Another well-known example is the construction of the curvature tensors on a (pseudo-)Riemannian manifold via the Levi--Civita connection. In physics, connections are central to the notion of gauge fields such as the electromagnetic field. Connections also play a   r\v{o}le in geometric approaches to relativistic mechanics, Fedosov's approach to deformation quantisation, adiabatic evolution via the Berry phase, and so on. For a review of the impact of the theory of connections on classical and quantum field theory the reader my consult \cite{Mangiarotti:2000}.  In the context of classical mechanics and control theory, connections appear in certain problems related to reductions by symmetry groups and systems with constraints (see for example \cite{Block:1996, Cardin:1996,Marsden:1992}).\par
Also prevalent throughout differential geometry --  again in many forms -- are Lie algebroids as first described by Pradines \cite{Pradines:1974}. Loosely, a Lie algebroid can be viewed as a mixture of tangent bundles and Lie algebras. A little more carefully, a Lie algebroid is a vector bundle $\pi : A \rightarrow M$, that comes equipped with a Lie bracket on the space of sections, together with an anchor map $\rho: \Sec(A) \rightarrow \Vect(M)$, that satisfy some natural compatibility conditions. In particular, the anchor map is a Lie algebra homomorphism. With this in mind, there is a general mantra: \emph{whatever you can do with tangent bundles you can do with Lie algebroids}. Our general reference for the theory of Lie algebroids and Lie groupoids is the book by Mackenzie \cite{Mackenzie:2005}.\par 
Given a Lie algebroid $(A, [-,-], \rho )$, the notion of an \emph{$A$-connection}  on a general fibre bundle was first given by Fernandes \cite{Fernandes:2002}. He put this notion to good use by discussing holonomy and characteristic classes of Lie algebroids. His approach was to define a generalisation of a horizontal distribution in which in part of the definition a Lie algebroid is used rather than a tangent bundle. The notion of a \emph{linear $A$-connection} on a vector bundle is clear: the resulting connection must respect the linear structure. \par 
In this paper, we examine the related concept of a connection on a graded bundle (see Grabowski \& Rotkiewicz \cite{Grabowski:2012}) that is in a precise sense adapted to the graded structure. We will call such connections \emph{weighted $A$-connections} as we allow them to take their values in Lie algebroids.\par 
Graded bundles are a particular `species' of non-negatively graded manifold (see Voronov \cite{Voronov:2001}), and one that has a very well behaved mathematical structure.  Without details, a graded bundle is a fibre bundle (in the category of smooth manifolds) for which one can assign a  weight of zero to the base coordinates and a non-zero weight to the fibre coordinates. Admissible changes of local coordinates respect this assignment of weight. It is an amazing fact that the structure of a graded bundle $F$, is completely  encoded in a smooth action of the multiplicative reals $\rmh : \R \times F \rightarrow F$, a \emph{homogeneity structure} in the language of Grabowski \& Rotkiewicz. Other than smooth, there is no further condition on this action.  Principal examples of graded bundles include higher order tangent bundles and vector bundles. Indeed, a vector bundle can be considered as a graded bundle by assigning a weight of zero to the base coordinates and a weight of one to the fibre coordinates. This assignment is respected changes of coordinates as all such transformation laws are linear in the fibre coordinates.  The ethos one can take is that graded bundles are `non-linear' vector bundles, and so the question of connections that in some sense respect the graded structure is a natural one. Moreover, as the r\^{o}le of graded bundles in geometric mechanics becomes more prevalent (see \cite{Bruce:2015,Grabowska:2014,Grabowska:2015}) the question of connections adapted to the structure of a graded bundle may become a practical one. \par 
Following  Va\v{\i}ntrob \cite{Vaintrob:1997}, we will understand a Lie algebroid to be a particular kind of Q-manifold, that is a supermanifold equipped with an odd vector field that `squares to zero'. This supermanifold picture of Lie algebroids will heavily influence our approach to connections. In particular, we will start from a definition of a weighted $A$-connection as a certain Grassmann odd vector field $\nabla$ on the supermanifold $\Pi A\times_M F$, build from a Lie algebroid $A$ and a graded bundle $F$ both over the same base manifold $M$. The key point is that a weighted $A$-connection by definition must be projectable to the homological vector field $\rmd_A \in \Vect(\Pi A)$ encoding the Lie algebroid structure. Moreover, the weight of a weighted  $A$-connection is zero with respect to the graded structure of $F$. This condition on the weight is enforced by the desire for the connection to respect the structure of the graded bundle. We will make these statements more precise in due course.  If the weighted $A$-connection is `homological', then we have a \emph{flat weighted $A$-connection}.  For the case of a  linear connection on a vector bundle, this `odd vector field' approach is a slight reformulation of  the notion of a Koszul connection. When the graded bundle in question is a vector bundle and $A$ is any Lie algebroid over the same base manifold, we recover the notion of  linear $A$-connections as given by Fernandes \cite{Fernandes:2002}, also see Vitagliano  \cite[Example 6.]{Vitagliano:2016} where a Q-manifold approach to flat connections is presented. As another example, we recover the notion of a \emph{homogeneous nonlinear connection} when the graded bundle is a higher order tangent bundle and the Lie algebroid is  the tangent bundle (see \cite{Andres:1989}). The general situation regarding weighted $A$-connections can be seen as a mixture of these two extreme examples. \par 
We show, by  using the results of K\v{r}i\v{z}ka \cite{Krizka:2008} on the existence of linear $A$-connections, that weighted $A$-connections exist for any Lie algebroid and any graded bundle over the same base manifold. The methodology is to use the fact that all graded bundles are non-canonically isomorphic to split graded bundles, i.e., a Whitney sum of graded vector  bundles (see \cite{Bruce:2016}). One can then consider splittings of graded bundles as \emph{gauge transformations} and use this to \emph{define} a weighted $A$-connection given a linear $A$-connection on the associated Whitney sum of graded vector bundles. This construction leads to the fact  that all weighted $A$-connections are gauge equivalent to linear $A$-connections on a split graded bundle. Naturally, we will make this precise in due course.\par 
It is well-known that \emph{all} fibre bundles admit connections. This follows from the fact that connections can be identified with sections of the first jet bundle of the fibre bundle under study, and as this is an affine bundle, sections always exist. Thus, considered as just a fibre bundle, we know that a graded bundle always admits a connection.  However, this does not tell us that connections that respect some additional structure, in this case, the graded structure,  exist. For example, it  is well-known that for a holomorphic vector bundle over a complex manifold, the Atiyah class is the obstruction to the existence of a holomorphic connection, see Atiyah \cite{Atiyah:1957}. However, considered as just a vector bundle,  we know that linear connections always exist.  \par 
Passing from the notion of a  weighted $A$-connection on a graded bundle to a \emph{multi-weighted $A$-connection} on a $n$-fold graded bundle is only a matter of taking into account the additional  compatible homogeneity structures. Following Grabowski \& Rotkiewicz \cite{Grabowski:2009} (also see Voronov \cite{Voronov:2001}), we know that the structure of a double vector bundle can be described in terms of a pair of mutually commuting (regular) homogeneity structures. This reformulation is very economical and workable  as compared with the original  definition of Pradines \cite{Pradines:1974} in terms of a pair of addition operations that are compatible in a specific sense. In particular, using double graded bundles to capture the structure of a double vector bundle allows for a notion of a connection adapted to the structure in much the same way as a linear connection is  adapted to the structure of a vector bundle. We will refer to these adapted connections as \emph{bi-weighted $A$-connections} on double vector bundles.\par 
Canonical examples of double vector bundles include the tangent and cotangent bundle of vector bundles and so such objects naturally appear in differential geometry. Moreover, double vector bundles play a fundamental r\^{o}le in geometric mechanics in the spirit of Tulczyjew and generalisations thereof to Lie algebroids,   see for example  Grabowska, Grabowski \& Urba\'{n}ski  \cite{Grabowska:2006}. Connections are typically needed when dealing with time-dependent mechanics and relativistic systems where they essentially correspond to choosing an observer. Thus, it is  plausible that connections adapted to double vector bundles could be of use in geometric mechanics.  \par 
As a general remark, it seems that the power and elegance of describing Lie algebroids as Q-manifolds   has not been fully utilised in the existing literature. The descriptions of vector bundles and double vector bundles in terms of (regular) homogeneity structures have also not been widely employed. Both of these ideas are central to this paper.

\smallskip 

\noindent \textbf{Main Results.} Amongst other  results contained within this paper we have:
\begin{itemize}
\item Theorem \ref{thm:existence connections} that states the existence of weighted $A$-connections for any Lie algebroid and graded bundle over the same base manifold. Proposition \ref{prop:gaugeequiv} and Corollary \ref{coro:gaugeequiv} tell us that weighted $A$-connections on a graded bundle are gauge equivalent to linear $A$-connections on the corresponding split graded bundle, i.e., the associated Whitney sum of graded vector bundles.
\item Theorem \ref{thm:actions} that states the one-to-one correspondence between flat weighted $A$-connections and infinitesimal actions of Lie algebroids on graded bundles that respect the graded structure.
\item Theorem \ref{thm:existence bi connections} that states the existence of  bi-weighted $A$-connections adapted to the structure of double vector bundles.
\end{itemize}

\smallskip
\noindent \textbf{Impetus.} This work arose from  discussions with Janusz Grabowski and Luca Vitagliano in regards to the relation between weighted Lie algebroids and representations up to homotopy of Lie algebroids. These discussions accumulated to the paper  \cite{Bruce:2017}. It was early in this collaboration that it was realised that phrasing connections in terms of odd vector fields, so something akin to weakening action Lie algebroids, was an economical and unifying framework well suited to the category of graded bundles. The question of the existence of connections adapted to the structure of a graded bundle was also raised in informal discussions with Miko{\l}aj Rotkiewicz. 

\smallskip
\noindent \textbf{Possible Generalisations.} There are at least two clear directions for possible future work:
\begin{enumerate}
\item connections on filtered bundles, i.e., polynomial bundles with permissiable coordinate changes being filtered rather than graded  (see \cite{Bruce:2018}), and 
\item connections with values in weighted Lie algebroids (see \cite{Bruce:2014a,Bruce:2015,Bruce:2016}).
\end{enumerate}
As examples of both filtered bundles and weighted Lie algebroids in relation to geometric mechanics and classical field theory are plentiful, the study of connections in this context could prove to be rich. 

\smallskip

\noindent\textbf{Arrangement.} In Section \ref{sec:preliminaries} we recall the basics of the theory of supermanifolds, graded bundles and Lie algebroids as needed throughout the rest of this paper.  The informed reader may safely skip this section. The bulk of the work is to be found in Section \ref{sec:WeightedCon} were new results are presented. It is in this section that the notion of a  weighted $A$-connection is given and their existence is established. We conclude this paper in Section \ref{sec:MultiWeightedCon} where we modify the notion of a connection to the setting of multi-graded bundles, examples of which include double vector bundles.

\section{Graded bundles and Lie algebroids}\label{sec:preliminaries}
\subsection{Supermanifolds}
 We understand a \emph{supermanifold} $\mathcal{M} := (|\mathcal{M}|, \:  \cO_{\mathcal{M}})$ of dimension $n|m$ to be a locally superringed space that is locally isomorphic to $\mathbb{R}^{n|m} := \big (\R^{n}, C^{\infty}(\R^{n})\otimes \Lambda(\theta^{1}, \cdots, \theta^{m}) \big)$. In particular,  given any point on $|\mathcal{M}|$ we can always find a `small enough' open neighbourhood $|U|\subseteq |\mathcal{M}|$ such that we can  employ local coordinates $x^{A} := (x^{a} , \theta^{\alpha})$ on $\mathcal{M}$, where $x^{a}$  and $\theta^\alpha$ are, respectively,  commuting and anticommuting coordinates. We will call (global) sections of the structure sheaf \emph{functions}, and often  denote the supercommutative algebra of all functions as $C^{\infty}(\mathcal{M})$. \emph{Morphisms of supermanifolds} are morphisms of locally superringed spaces. The underlying smooth manifold $|\mathcal{M}|$ we refer to as the \emph{reduced manifold}.    In particular, morphisms preserve the $\mathbb{Z}_2$-grading of the structure sheaf. The Grassmann parity of an object $O$  will be denoted by `tilde', i.e., $\widetilde{O} \in \mathbb{Z}_2$.  By `even' and `odd' we will be referring the Grassmann parity of  the objects in question.  Importantly, we have a \emph{chart theorem} that allows us to (locally) describe morphisms of supermanifolds in terms of local coordinates in much the same way as one can on smooth manifolds. We will make heavy use of local coordinates and employ the standard abuses of notation when it comes to describing morphisms. 
\begin{example}
Consider a vector bundle in the category of smooth manifolds  $\pi : E \rightarrow M$. We can employ adapted local coordinates of the form  $(x^a, u^i)$ together with the admissible changes of local coordinates
\begin{align*}
& x^{a'} = x^{a'}(x), && u^{i'} = u^j T_j ^{\:\: i'}(x).
\end{align*} 
From this datum, we can construct a supermanifold by applying the \emph{parity reversion functor} $\Pi$. The obtained supermanifold $\Pi E$, which we will refer to as an \emph{antivector bundle}, comes equipped with local coordinates $(x^a, \zx^i)$, with the same admissible changes as above, but now $\zx^i$ are anticommuting.
\end{example}
The \emph{tangent sheaf} $\mathcal{T}\mathcal{M}$ of a supermanifold $\mathcal{M}$ is the sheaf of derivations of sections of the structure sheaf -- this is, of course, a sheaf of locally free $\cO_{\mathcal{M}}$-modules. Sections of the tangent sheaf we refer to as \emph{vector fields} and denote the $C^\infty(\mathcal{M})$-module of vector fields as $\Vect(\mathcal{M})$. Under the $\mathbb{Z}_2$-graded commutator bracket the space of vector fields on a supermanifold forms a Lie superalgebra, i.e., we have 
\begin{align*}
& [X,Y] = {-}(-1)^{\widetilde{X} \widetilde{Y}} [Y,X],\\
& [X, [Y,Z]] = [[X,Y],Z] + (-1)^{\widetilde{X} \widetilde{Y}} [Y,[X,Z]].
\end{align*}
Importantly, vector fields on a supermanifold can be `localised' in the sense that they can be written using local coordinates in terms of partial derivatives, i.e., 
$$X|_{|U|} = X^A(x)\frac{\partial}{\partial x^A},$$
in much the same way as one would write a vector field on a manifold. We will generally drop the restriction of objects to open subsets of the reduced manifold when writing things locally in terms of coordinates.

\begin{definition}[Adapted from \cite{Alexandrov:1997}]
A \emph{Q-manifold}  is a supermanifold  $\mathcal{M}$, equipped with a distinguished  odd vector field $Q\in \Vect(\mathcal{M})$  that `squares to zero', i.e., $Q^{2} = \half [Q,Q] =0$. The vector field $Q$ is referred to as a \emph{homological vector field} or a \emph{Q-structure}. A \emph{morphism of Q-manifolds} $\phi: (\mathcal{M}, Q) \rightarrow (\mathcal{M}^\prime , Q^\prime)$ is a morphism of supermanifolds that furthermore relates the homological vector fields, i.e.,
$$Q \circ \phi^*  = \phi^* \circ Q^\prime.$$
\end{definition}
For more details of supermanifolds, the reader may consult, for instance, Carmeli, Caston \& Fioresi \cite{Carmeli:2011},  Manin \cite{Manin:1997}  and Varadarajan \cite{Varadarajan:2004}. For a comparison of the `algebro-geometric approach' with the `concrete approach', the reader may consult Rogers \cite{Rogers:2007}.
\subsection{Graded bundles}
For a short review of the theory of graded bundles the reader my consult \cite{Bruce:2017a}. Our general understanding of graded (super)manifolds will be in the sense of  Voronov \cite{Voronov:2001} (also see \v{S}evera \cite{Severa:2005} and Roytenberg \cite{Roytenberg:2001}). We will  only  consider manifolds $F$ that carry a non-negative grading that is associated with a smooth action $\rmh:\R\ti F\to F$ of the monoid $(\R,\cdot)$ of multiplicative reals. We will not explicitly discuss supermanifolds in this subsection.  Such actions   are referred  to as  \emph{homogeneity structures} (see Grabowski \& Rotkiewicz \cite{Grabowski:2012}). This action reduced to $\R_{>0}$ is the one-parameter group of diffeomorphism integrating the \emph{weight vector field}, thus the weight vector field is, in this case, \emph{h-complete}  and only \emph{non-negative integer weights} are allowed (see \cite{Grabowski:2013,Grabowski:2012}). Thus the algebra $\mathcal{A}(F)\subset C^\infty(F)$ spanned by homogeneous functions is $\mathcal{A}(F) =  \bigoplus_{i \in \mathbb{N}}\mathcal{A}^{i}(F)$, where $\mathcal{A}^{i}(F)$ consists of homogeneous functions of degree $i$. We will denote the weight of a homogeneous function $f \in \mathcal{A}^{i}(F)$ (and similar for other geometric objects) as $\w(f) = i$. \par
For $t \neq 0$ the action $\rmh_{t}$ is a diffeomorphism of $F$ and, when $t=0$, it is a smooth surjection $\tau=\rmh_0$ onto $F_{0}=:M$, with the fibres being diffeomorphic to $\R^N$, for the appropriate $N \in \mathbb{N}$.  Thus, the smooth manifolds obtained are particular kinds of \emph{polynomial bundles} $\tau:F\to M$, i.e., fibrations which locally look like $U\times\R^N$    ($U \subset M$ open) and the change of coordinates (for a certain choice of an atlas) are polynomial in $\R^N$. For this reason graded manifolds with non-negative weights \emph{and} h-complete weight vector fields are also known as \emph{graded bundles}.
\begin{example}\label{e1}
Consider a manifold $M$ and $\dd=(d_1,\dots,d_k)$, with positive integers $d_i$. The trivial fibration $\tau:M\times\R^\dd\to M$, where
$\R^\dd=\R^{d_1}\times\cdots\times\R^{d_k}$, is canonically a graded bundle with the homogeneity structure $\rmh^\dd$ given by $\rmh_t(x,y)=(x,\rmh_t^\dd(y))$, where
\begin{equation}\label{hsm}
\rmh_t^\dd(y_1,\dots,y_k)=(t\cdot y_1,\dots, t^k\cdot y_k)\,,\quad y_i\in\R^{d_i}\,.
\end{equation}
\end{example}
\begin{theorem}[Grabowski--Rotkiewicz \cite{Grabowski:2012}]\label{theorem:1}
Any graded bundle $(F,\rmh)$ is a locally trivial fibration $\tau:F\to M$ with a typical fibre $\R^\dd$, for some $\dd=(d_1,\dots,d_k)$, and the homogeneity structure locally equivalent to the one in Example \ref{e1}, so that the transition functions are graded isomorphisms of $\R^\dd$. In particular, any graded space is diffeomorphically equivalent with $(\R^\dd,\rmh^\dd)$ for some $\dd$. 
\end{theorem}
\begin{remark}
The equivalence of homogeneity structures acting on supermanifolds and graded super bundles is established in J\'{o}\'{z}wikowski \& Rotkiewicz \cite{Jozwikowski:2016}.  
\end{remark}
It follows that on  a  graded bundle, one can always find an  atlas of $F$  consisting of charts for which we  have homogeneous local coordinates $(x^a, y_w^I)$, where $\w(x^a) =0$ and  $\w(y_w^I) = w$ with $1\leq w\leq k$, for some $k \in \mathbb{N}$ known as the \emph{degree}  of the graded bundle.  Note that a graded bundle of degree $k$ is,  by definition,  also  a graded bundle of degree $l$ for $l\ge k$. However, there is always a \emph{minimal degree} and this is what we usually mean by degree.  The index  $I$ should be considered as a ``generalised index" running over all the possible weights. The label $w$ is largely redundant, but it will come in very useful when checking the weight of various (local) expressions.  The local changes of coordinates  respect the weight and hence are polynomial for non-zero weight coordinates, thus they are of the form:
\begin{eqnarray}\label{eqn:translawsGr}
x^{a'} &=& x^{a'}(x),\\
\nonumber y^{I'}_{w} &=& y^{J}_{w} T_{J}^{\:\: I'}(x) + \sum_{\stackrel{1<n  }{w_{1} + \cdots + w_{n} = w}} \frac{1}{n!}y^{J_{1}}_{w_{1}} \cdots y^{J_{n}}_{w_{n}}T_{J_{n} \cdots J_{1}}^{\:\:\: \:\:\:\:\: I'}(x),
\end{eqnarray}
where $T_{J}^{\:\: I'}$ are invertible and the $T_{J_{n} \cdots J_{1}}^{\:\:\: \:\:\:\:\:I'}$ are symmetric in lower indices.\par
A graded bundle  of   degree $k$  admits a sequence of  surjections (see Voronov \cite[Remark 4.2.]{Voronov:2001})
\begin{equation}\label{eqn:fibrations}
F:=F_{k} \stackrel{\tau^{k}_{k-1}}{\longrightarrow} F_{k-1} \stackrel{\tau^{k-1}_{k-2}}{\longrightarrow}   \cdots \stackrel{\tau^{3}_2}{\longrightarrow} F_{2} \stackrel{\tau^{2}_1}{\longrightarrow}F_{1} \stackrel{\tau^{1}}{\longrightarrow} F_{0} =: M,
\end{equation}
\noindent where $F_l$ itself is a graded bundle over $M$ of degree $l$ obtained from the atlas of $F_k$ by removing all coordinates of degree greater than $l$ (see the next paragraph). \par
Note that  $F_{1} \rightarrow M$ is a linear fibration and the other fibrations $F_{l} \rightarrow F_{l-1}$ are affine fibrations in the sense that the changes of local coordinates for the fibres are linear plus  an additional additive terms of appropriate weight.  The model fibres here are $\mathbb{R}^{n}$ for some $n \in \mathbb{N}$.\par
\begin{example}\label{exm:highertangent} 
The fundamental  example of a graded bundle is the higher tangent bundle $\sT^{k}M$, i.e., the $k$-th jets (at zero) of curves $\gamma: \mathbb{R} \rightarrow M$.
Given a smooth function $f$ on a manifold $M$, one can construct functions $f^{(\alpha)}$ on $\sT^k M$, where $0\leq \alpha\leq k$, which are referred to as the $(\alpha)$-lifts of $f$ (see \cite{Morimoto:1970}). These functions are  defined by
$$
f^{(\alpha)}([\gamma]_k):= \left.\frac{\rmd^\alpha}{\rmd t^\alpha}\right|_{t=0} f(\gamma(t)),
$$
where $[\gamma]_k\in \sT^k M$ is the class of the curve $\gamma:\R\rightarrow M$.
The smooth  functions $f^{(k)} \in C^\infty\big(\sT^k M \big)$ and $f^{(1)} \in C^\infty\big(\sT M \big)$ are called the $k$-complete lift and the tangent lift of $f$, respectively. Coordinate systems $(x^a)$ on $M$ gives rise to so-called adapted or homogeneous coordinate systems  $(x^{a, (\alpha)})_{0\leq \alpha \leq k}$ on $\sT^k M$ in which $x^{a, (\alpha)}$ is of weight $\za$. Fa\`{a} di Bruno's formula, i.e., repeated application of the chain rule, shows that the admissible changes of adapted coordinates are of the form \eqref{eqn:translawsGr}. The homogeneity structure on $\sT^{k}M$ can then be defined via local coordinates.
\end{example}
\begin{example}[Taken from \cite{Grabowska:2014}]
If $\pi : E \rightarrow M$ is a vector bundle, then $\bigwedge^k \sT E$ is a graded bundle of degree $k$ with respect to the projection $\bigwedge^k \sT \pi :\bigwedge^k \sT E \rightarrow \bigwedge^k \sT M$. For the $k=2$ case, we can employ homogeneous local coordinates  $(x^a, \dot{x}^{bc}, y^\alpha, y^{d \beta}, z^{\delta \gamma})$, where $\dot{x}^{ab} = - \dot{x}^{ba} $ and $z^{\alpha \beta} = - z^{\beta \alpha}$, of weight $0,0,1,1$ and $2$, respectively.   These coordinates correspond to the local decomposition of a bi-vector field as 
$$X = \frac{1}{2} \dot{x}^{ab} \frac{\partial}{\partial x^b} \wedge \frac{\partial}{\partial x^a} + y^{a \alpha} \frac{\partial}{\partial y^\alpha} \wedge \frac{\partial}{\partial x^a} + z^{\alpha \beta} \frac{\partial}{\partial y^\beta} \wedge \frac{\partial}{\partial y^\alpha} \,.   $$
The invariance of this local decomposition induces the required changes of coordinates on $\bigwedge^2 \sT E$.
\end{example}
\begin{remark}
 \v{S}evera \cite{Severa:2005} and Roytenberg \cite{Roytenberg:2001} defined  $\N$-manifolds in two different, but ultimately equivalent  ways. \v{S}evera's definition uses an $(\R ,\cdot)$ action such that $-1 \in \R$ acts as the parity operator. Roytenberg  defines an $\N$-manifold as a graded supermanifold (see \cite{Voronov:2001}) for which the Grassmann parity of the local coordinates is given by their weight mod $2$. Note that higher tangent bundles are \emph{not} $\N$-manifolds, nor can they directly be `superised' (see \cite{Bruce:2016b}). The reason we prefer to work with graded bundles rather than $\N$-manifolds is due to the fact that  higher tangent bundle are natural examples of the former. 
\end{remark}

A homogeneity structure is said to be \emph{regular} if
\begin{equation}\label{eqn:HomoReg}
\left.\frac{\rmd }{\rmd t}\right|_{t=0}\rmh_{t}(p) = 0   \hspace{25pt}\Rightarrow   \hspace{25pt}  p = \rmh_{0}(p),
\end{equation}
for all points $p \in F$. Moreover, if homogeneity structure is regular then the graded bundle has the structure of a vector bundle (see \cite{Grabowski:2009}). In the converse direction, in any chart adapted to the vector bundle structure we may assign weight zero to the base coordinates and weight one to the fibre coordinates. As the permissible changes of the linear coordinates are linear, this assignment of weight is preserved.  In fact, we need not assign a weight one  to the fibre coordinates, any positive integer weight will suffice. \par 
Morphisms between graded bundles are morphisms of smooth manifolds that   preserve the assignment of the weight. In other words, morphisms relate the respective homogeneity structures. More carefully, if we have two graded bundles (not necessarily of the same minimal degree)  $(F, \rmh)$ and $(F^\prime, \rmh^\prime)$, then a morphism between the two graded bundles is a morphism of smooth manifolds $\phi: F \rightarrow F^\prime$ that satisfies
$$\phi \circ \rmh_t  = \rmh^\prime_t \circ \phi,$$
for all  $t \in \R$. Evidently, morphisms of graded bundles can be composed as standard morphisms between smooth manifolds and so we obtain the \emph{category of graded bundles}. 
\begin{definition}
A \emph{split graded bundle} is a graded bundle of the form
$$F_k^{\:\textnormal{split}} :=E_{1} \times_{M}E_{2}\times_{M} \cdots  \times_{M} E_{k}, $$
where each $E_i \rightarrow M$ is a graded vector bundle of degree $i \in \mathbb{N}$, i.e, the fibre coordinates are assigned  the weight $i$. The homogeneity structure is provided by the natural homogeneity structure on the Whitney sum of graded vector bundles.
\end{definition}
We have a version of the famous Batchelor--Gaw\c{e}dzki theorem for graded bundles.
\begin{theorem}[Bruce-Grabowska-Grabowski \cite{Bruce:2016}]\label{theorem:splitting}
Any graded bundle $(F_k, \rmh)$ is non-canonically isomorphic (in the category of graded bundles)  to a split graded bundle $F_k^{\:\textnormal{split}}$. 
\end{theorem}
The above theorem says that we can always find an isomorphism in the category of graded bundles, i.e., a \emph{splitting}:
$$\varphi : F_k \longrightarrow F_k^{\:\textnormal{split}}.$$
However, a splitting is never unique and rarely is it canonical. Each $E_i$ we refer to as a `building vector bundle' for $F_k$. 
\begin{remark}
An identical statement can be made for $\N$-manifolds. While this has been `folklore' for a while, the first proof  to appear in the literature that we are aware of is that of Bonavolont\`{a} \&  Poncin  \cite{Bonavolonta:2013}.
\end{remark}
\begin{example}
The $k$-th order tangent bundle $\sT^k M$ is \emph{non-canonically} isomorphic   to  
$$\sT^{\times_k}M := \underbrace{\sT M \times_M \sT M \times_M \cdots \times_M \sT M}_{\textnormal{k times}}\,,$$
equipped with its natural regular homogeneity structure, i.e., we assign a weight of $i$ to  fibre coordinates of the $i^{\textnormal{th}}$ factor in the fibre product.  The specific case of 
$$\sT^2 M \xrightarrow{\sim} \sT M \times_M \sT M\,,$$
is well-known and corresponds to the specification of an affine connection on $M$. Naturally, without any further structure, there is no way to single out a particular affine connection as being in anyway privileged.
\end{example}
\begin{example}
Consider a vector bundle $\pi: E \rightarrow M$. The associated vertical bundle $\sV E \subset \sT E$, which is naturally considered as a graded bundle of degree $2$,  is \emph{canonically} isomorphic to $E \times_M E$. This example is well-known and is easily justified via adapted local coordinates. The point is that the isomorphism does not require any additional structure to be specified, this is in stark contrast to the previous example.
\end{example}
The notion of a \emph{double vector bundle}   is conceptually  clear  in terms of mutually  commuting regular homogeneity structures (see  \cite{Grabowski:2009, Grabowski:2012}). The original notion as given by Pradines \cite{Pradines:1974} (also see \cite{Koneczna:1999,Mackenzie:1992}) is much more complicated involving various compatibility conditions between the two vector bundle structures. We will discuss double vector bundles in more detail in Section \ref{sec:MultiWeightedCon}.  Relaxing the condition that the homogeneity structures be regular leads to the notion of a \emph{double graded bundle}.  In this paper we will encounter double graded bundles for which one of the homogeneity structures is regular, these particular double graded bundles were referred to as \emph{graded--linear bundles} in \cite{Bruce:2016}. 
\begin{remark}
We will \emph{not} consider more general graded geometries that require local coordinates of both positive and negative degrees. From our perspective, such geometries exhibit  pathological behaviour. Specifically, for general $\mathbb Z$-graded manifolds, the homogeneous functions are power series in the local coordinates and not just polynomials; there is no canonical bundle structure over the underlying base manifold; and it is not known if there is an analogue of the Batchelor--Gaw\c{e}dzki theorem.  For a review of $\mathbb Z$-graded geometry, the reader may consult Fairon \cite{Fairon:2017}.
\end{remark}

\subsection{Lie algebroids as Q-manifolds}
Our philosophy is that a Lie algebroid \emph{is} a Q-manifold equipped with a regular homogeneity structure that is compatible with the Q-structure. Let $\pi : A \to M$ be a  vector bundle. In particular, $A$ comes equipped with a regular homogeneity structure 
$$\rml : \R \times A  \rightarrow A, $$
 and we can make the identification $\pi = \rml_0$. This means that  we can assign a weight of zero to the base coordinates and a weight of one to the fibre coordinates. One can then apply the parity reversion functor and obtain  $\Pi A$, which naturally comes equipped with a regular homogeneity structure inherited from that on $A$.  We will denote the homogeneity structure on the parity reversed vector bundle by $\rml$ also, this should cause no confusion. Following Va\v{\i}ntrob \cite{Vaintrob:1997} it is well known that a Lie algebroid structure on a vector bundle $A$  is equivalent to a homological vector field of weight one on  $\Pi A$. 
\begin{definition}
A \emph{Lie algebroid} is a Q-manifold  equipped with a regular homogeneity structure $\rml : \R \times \Pi A \rightarrow \Pi A$, such that the  homological vector field $\rmd_A \in \Vect(\Pi A)$, is of weight one. That is
$$\rml_{t}^{*}\circ \rmd_{A} = t \: \rmd_{A} \circ \rml_{t}^{*},$$
for all $t \in \R$.
\end{definition}
By convention, we will denote a Lie algebroid by $(\Pi A, \rmd_{A}, \rml)$, or simply $(\Pi A, \rmd_{A})$. In natural local coordinates $(x^{a}, \zx^{i})$, the  homological vector field is given by
$$\rmd_{A} = \zx^{i}Q_{i}^{a}(x) \frac{\partial}{\partial x^{a}} + \frac{1}{2!} \zx^{i}\zx^{j}Q_{ji}^{k}(x)\frac{\partial}{\partial \zx^{k}}.$$
In these local coordinates, we have the  \emph{Lie algebroid structure equations}
\begin{align*}
& Q_{ji}^{k}Q_{k}^{b} = Q_{j}^{a}\frac{\partial Q_{i}^{b}}{\partial x^{a}} {-}  Q_{i}^{a} \frac{\partial Q_{j}^{b}}{\partial x^{a}},
&& \sum_{\textnormal{cyclic}} \left(Q_{i}^{a}\frac{\partial Q_{jk}^{l}}{\partial x^{a}}  +  Q_{ij}^{m}Q_{mk}^{l}\right) =0.
\end{align*}

The notion of a \emph{morphism of Lie algebroids} is clear. Suppose we have two Lie algebroids $(\Pi A, \rmd_A , \rml)$ and  $(\Pi A^\prime, \rmd_{A^\prime} , \rml^\prime)$,  not necessarily over the same base manifold, then a morphism between them is a morphism of supermanifolds
$$\phi:  \Pi A \longrightarrow \Pi A^\prime,$$
that respects the additional structures, i.e.,
\begin{align*}
&  \phi \circ \rml_t = \rml_t^\prime \circ \phi,
&& \rmd_A \circ \phi^* = \phi^* \circ  \rmd_{A^\prime}.
\end{align*}
Evidently, we can compose morphisms of Lie algebroids and thus we obtain the \emph{category of Lie algebroids}.\par 

The derived bracket formalism of Kosmann-Schwarzbach (see \cite{Kosmann-Schwarzbach:1996,Kosmann-Schwarzbach:2003}) allows us to translate between the definition of a Lie algebroid in terms of a Q-manifold and that in terms of a bracket and anchor.  Note that we have an odd linear mapping $\Sec(A) \longrightarrow \Vect(\Pi A)$
given by
$$u = u^{i}(x)e_{i} \leftrightsquigarrow \iota_{u} :=  u^{i}(x)\frac{\partial}{\partial \zx^i} $$ for any section $u \in \Sec(A)$. Clearly, we have an isomorphism of the spaces taking into account the shift in Grassmann parity. The \emph{Lie algebroid bracket} is defined as
$$\iota_{[u,v]_A} = [[\rmd_A,\iota_{u}], \iota_{v}],$$
for any sections $u$ and $v   \in \Sec(A)$. The \emph{anchor map} $\rho: \Sec(A) \rightarrow \Vect(M)$ is defined as
$$\rho_{u}(f) :=  [[\rmd_A,\iota_{u}],f],$$
for any function $f \in C^{\infty}(M)$. Explicitly in a chosen local basis of sections of $A$ we have
\begin{align*}
& [u,v]_A = \left (u^i Q_i^a \frac{\partial v^k}{\partial x^a} {-}  v^i Q_i^a \frac{\partial u^k}{\partial x^a}   {-} u^i v^j Q_{ji}^k\right)e_k,
&& \rho_u(f) = u^i Q_i^a\frac{\partial f}{\partial x^a}.
\end{align*}
 It is well known that we have the following properties  of the bracket and anchor
\begin{enumerate}
\item $[u,v]_A = {-} [v,u]_A$;
\item $[u, fv]_A = \rho_{u}(f) v +   f [u,v]_A$;
\item $[u,[v,w]_A]_A = [[u,v]_A,w]_A + [v, [u,w]_A]_A$;
\item $[\rho_{u}, \rho_{v}] - \rho_{[u,v]_A} =0$.
\end{enumerate}
\begin{example}
The antitangent bundle $\Pi \sT M$ comes equipped with the de Rham differential. In natural coordinates $(x^{a}, \rmd x^{b})$ the de Rham differential is given by
$$\rmd =  \rmd x^{a} \frac{\partial}{\partial x^{a}}.$$ 
The associated derived bracket is just the standard Lie bracket on vector fields on $M$ and the anchor is the identity map.
\end{example}
\begin{example}
A super-vector space $U = \Pi \mathfrak{g}$ considered as a linear supermanifold, equipped with a homological vector field of weight one is equivalent to a Lie algebra. This is really no more than a minor reformulation of the Chevalley--Eilenberg algebra of a Lie algebra.  In natural adapted coordinates $\zx^{i}$ the differential is given by
$$\rmd_\mathfrak{g} =\frac{1}{2!} \zx^{i}\zx^{j}Q_{ji}^{k}\frac{\partial}{\partial \zx^{k}}.$$
The associated derived bracket is the Lie bracket on $\mathfrak{g}$.
\end{example}
We can also understand  the anchor as a morphism of graded super bundles
$$\rho^\Pi : \Pi A \longrightarrow \Pi \sT M,$$
which is given in local coordinates as
$$\big(x^a, \rmd x^b \big)\circ \rho^\Pi  :=  \big(x^a , \zx^iQ_i^a(x) \big).$$
General Lie algebroids can be viewed as a mixture of the two extreme examples presented above, for more details the reader can consult Mackenzie \cite{Mackenzie:2005}. A rather non-comprehensive list of examples includes  integrable distributions in a tangent bundle, the cotangent bundle of a Poisson manifold and bundles of Lie algebras.

\section{Connections adapted to graded bundles}\label{sec:WeightedCon}
\subsection{Weighted  $A$-connections}
Consider a Lie algebroid $(\Pi A, \rmd_{A}, \rml)$ and a graded bundle $(F_{k} , \rmh)$, both over the same base  $F_0 =: M$. We can then construct the pullback bundle $\tau^* \Pi A \simeq \Pi A\times_{M} F_k$, ($\tau :=\rmh_0$), together with  the following  commutative diagram
\begin{center}
\leavevmode
\begin{xy}
(0,20)*+{\tau^* \Pi A}="a"; (25,20)*+{\Pi A}="b";
(0,0)*+{F_{k}}="c"; (25,0)*+{M}="d";%
{\ar "a";"b"}?*!/_3mm/{\rmp_{\Pi A}};%
{\ar "a";"c"}?*!/^3mm/{\rmp_F};%
{\ar "b";"d"}?*!/_3mm/{\pi};%
{\ar "c";"d"}?*!/^3mm/{\tau};%
\end{xy}
\end{center}
We consider $\tau^* \Pi A$ as a bi-graded super bundle with the ordering of the homogeneity structures being $(\rmh, \rml)$. We can thus employ local coordinates
$$(\underbrace{x^{a}}_{(0,0)}, ~ \underbrace{\zx^{i}}_{(0,1)}, ~ \underbrace{y_{w}^{I}}_{(w,0)}),$$
where we have indicated the assignment of the bi-weight.
\begin{definition}\label{def:weighted connection}
With the above notation, a \emph{weighted $A$-connection} on a graded bundle $F_k$ is an odd vector field $\nabla \in \Vect\big(\tau^*\Pi A \big )$ of  bi-weight $(0,1)$ that projects to $\rmd_{A} \in \Vect(\Pi A)$. If the Lie algebroid is the tangent bundle, i.e, $A = \sT M$, then we simply speak of a \emph{weighted connection}.
\end{definition}
\begin{remark}
The similarity with Q-bundles, understood as fibre bundles in the category of Q-manifolds, following Kotov \& Strobl \cite{Kotov:2015} is clear. We are considering very particular bundle structures and are weakening the homological condition.
\end{remark}
Any weighted $A$-connection on $F_{k}$ must be of the local form
$$\nabla = \zx^{i}Q_{i}^{a}(x)\frac{\partial}{\partial x^{a}} + \frac{1}{2!}\zx^{i}\zx^{j}Q_{ji}^{k}(x)\frac{\partial}{\partial \zx^{k}} + \zx^{i}\Gamma_{i}^{I}[w](x,y)\frac{\partial}{\partial y^{I}_{w}}. $$
From here on we will use the notation $\Gamma_{i}^{I}[w]$ and similar to denote the homogeneous part of degree $w$ of an expression.  We will refer to the local components of a weighted $A$-connection as  \emph{Christoffel symbols}.\par 
The admissible changes of fibre coordinates on  $\tau^*\Pi A$ are of the form (see \eqref{eqn:translawsGr} )
\begin{align*}
&y_{w}^{I'} = T^{I'}[w](x,y), && \zx^{i'} = \zx^{j}T_{j}^{\:\:\: i'}(x). & 
\end{align*}
These changes of coordinates induce the transformation law for the Christoffel symbols
\begin{equation}\label{eqn:ChangeChristoffel}
 \Gamma^{I'}_{i'}[w] = T_{i'}^{\:\: j}\left(Q_{j}^{a}\frac{\partial T^{I'}[w]}{\partial x^{a}} + \Gamma_{j}^{J}[w_{0}]\frac{\partial T^{I'}[w]}{\partial Y_{w_{0}}^{J}} \right),
 \end{equation}
which  clearly respects the weight.
\begin{example}
Consider the trivial graded bundle $F_k := M \times \R^\dd$ (see Example (\ref{e1})) and any Lie algebroid $(\Pi A, \rmd_A)$ over $M$. In this case $\tau^* \Pi A \simeq \Pi A \times \R^\dd$. We can then have a canonical choice for a weighted $A$-connection, namely, we set $\nabla := \rmd_A$.  As with the classical case of connections on vector bundles, this choice is \emph{only} possible on trivial graded bundles.
\end{example}
\begin{example}
A linear $A$-connection on  a vector bundle $\tau : E \longrightarrow M$ is of the local form
$$ \nabla = \zx^{i}Q_{i}^{a}(x)\frac{\partial}{\partial x^{a}} + \frac{1}{2!}\zx^{i}\zx^{j}Q_{ji}^{k}(x)\frac{\partial}{\partial \zx^{k}} + y^\alpha \zx^{i}(\Gamma_i)_\alpha^{\:\: \beta}(x)\frac{\partial}{\partial y^{\beta}}\,, $$
where we have employed homogeneous local coordinates $(x^a, y^\alpha)$ of weight $0$ and $1$, respectively,  on $E$.
\end{example}
\begin{proposition}
Fix a Lie algebroid $\big( \Pi A, \rmd_A \big)$ and a graded bundle $\big( F_k, \rmh \big)$, both over a smooth manifold $M$. Let us assume that the set of weighted $A$-connections is non-empty. Then,  the set of weighted $A$-connections is an affine space modelled on vertical vector fields with respect to $\rmp_F : \tau^* \Pi A \rightarrow F_k$ of weight $(0,1)$.
\end{proposition}
\begin{proof}
The difference of two weighted $A$-connections is locally given by
$$\nabla {-}\nabla^\prime = \zx^i \left(\Gamma_i^I[w](x,y)  {- } \Gamma_i^{\prime I}[w](x,y)  \right) \frac{\partial}{\partial y_w^I}\,.$$
The transformation laws \eqref{eqn:ChangeChristoffel} show that we do indeed have a vector field in this way and inspection shows that it is vertical and weight $(0,1)$.
\end{proof}
We will denote the affine space of all weighted $A$-connections on a given graded bundle as $\mathcal{C}(F, \Pi A)$,  furthermore, we will prove this set is non-empty (see Theorem \ref{thm:existence connections}).
\begin{example}\label{exm:degree 2 christoffel}
To illustrate explicitly  the general situation let us consider a  graded bundle of degree $2$. We can employ local coordinates $(x^{a}, y^{\alpha}, z^{\mu})$ of weight $0,1$ and $2$ respectively.  The admissible changes of coordinates are
\begin{align*}
& x^{a'} = x^{a'}(x), && y^{\alpha'} = y^{\beta}T_{\beta}^{\:\:\: \alpha'}(x), && z^{\mu'} = z^{\nu}T_{\nu}^{\:\:\: \mu '}(x) + \frac{1}{2!}y^{\alpha}y^{\beta}T_{\beta \alpha}^{\:\:\:\:\: \mu'}(x).&
\end{align*}
 A weighted $A$-connection is then specified (locally) by the Christoffel symbols $\Gamma_{i}^{\alpha}[1]\big(x,y \big)$  and  $\Gamma_{i}^{\mu}[2]\big(x,y,z \big)$, which are polynomial in the non-zero weight coordinates. The transformation laws for the Christoffel symbols are
\begin{align*}
& \Gamma^{\alpha'}_{i'}[1] = T_{i'}^{\:\:\: j}\Gamma_{j}^{\beta}[1]T_{\beta}^{\:\:\: \alpha'} + y^{\beta}T_{i'}^{\:\:\: j}Q_{j}^{a}\frac{\partial T_{\beta}^{\:\:\: \alpha'}}{\partial x^{a}},\\
& \Gamma^{\mu'}_{i'}[2] = T_{i'}^{\:\:\: j}\Gamma_{j}^{\nu}[2]T_{\nu}^{\:\:\: \mu'} + y^{\beta}T_{i'}^{\:\:\:j}\Gamma_{j}^{\alpha}[1]T_{\alpha \beta}^{\:\:\:\:\: \mu'} + z^{\nu}T_{i'}^{\:\:\: j}Q_{j}^{a}\frac{\partial T_{\nu}^{\:\:\: \mu'}}{\partial x^{a}} + \frac{1}{2!}y^{\beta}y^{\alpha}T_{i'}^{\:\:\: j}Q_{j}^{a} \frac{\partial T_{\alpha \beta}^{\:\:\:\:\: \mu'}}{\partial x^{a}}.
\end{align*} 
\end{example}
\begin{definition}\label{def:FlatConnection}
A weighted $A$-connection is said to be a \emph{flat weighted $A$-connection} if and only if it is `homological', i.e.,
$$\nabla^2 = \frac{1}{2}[\nabla, \nabla] =0.$$
\end{definition}
\begin{remark}
Any flat weighted $A$-connection is equivalent to a \emph{weighted Lie algebroid} structure on $A \times_M F_k$, we certain properties (see \cite{Bruce:2014a, Bruce:2016}). That is, we have a classical  Lie algebroid  structure, i.e., bracket and anchor,  on $A \times_M F_k  \rightarrow F_k$ such that the bracket is of degree $-k$. In this case, we have a kind of action Lie algebroid, see Subsection \ref{subsec:Action}.
\end{remark}
In local coordinates, the \emph{curvature} of any weighted $A$-connection is easily shown to be given by 
\begin{align}\label{eqn:Curvature}
\nabla^2  & = {-}\frac{1}{2} \zx^i \zx^j \left( Q_j^a \frac{\partial \Gamma^I_i[w]}{\partial x^a}  \: {-}\:   Q_i^a \frac{\partial \Gamma^I_j[w]}{\partial x^a}  \: {-}\: Q_{ij}^k \Gamma_k^I[w] \right .\\
\nonumber  & + \left .  \Gamma_j^J[w_0] \frac{\partial \Gamma_i^I[w]}{\partial y^J_{w_0}} \: {-} \:   \Gamma_i^J[w_0] \frac{\partial \Gamma_j^I[w]}{\partial y^J_{w_0}} \right) \frac{\partial}{\partial y_w^I}
\end{align}
\noindent \textbf{Observations.} Given any graded bundle and Lie algebroid both over the same base, we have a series of affine bundle structures
$$\Pi A \times_M F_k  \xrightarrow{(\Id_{\Pi A} , \tau^k_l)} \Pi A \times_M F_l,  $$
where $l<k$ naturally  inherited from the graded structure of $F_k$ (see \eqref{eqn:fibrations}). Secondly, a weighted $A$-connection on $F_k$ is projectable to a weighted $A$-connection on $F_l$. To see this, it is clear that the Christoffel symbol $\Gamma_i^I[w_0]$ cannot depend on coordinates of weight greater than $w_0$. Thirdly, due to the properties of a pushforward, a flat weighted $A$-connection on $F_k$ is projectable to a flat weighted $A$-connection on $F_l$. \par 
From these observations, we are led to the following.
\begin{proposition}\label{prop:vector bundle connections}
Any weighted $A$-connection on a graded bundle $F_k$ gives rise to a series  of weighted $A$-connections on $\{F_l \}_{1 \leq l \leq k}$ induced  by the affine fibrations $\tau^k_l : F_k \rightarrow F_l$.  In particular, there is an induced  linear $A$-connection on the vector bundle $\tau^1:F_{1} \rightarrow M$.  Moreover, if the weighted $A$-connection is flat, i.e., $\nabla^{2} =0$, then the induced linear $A$-connection is also flat, and so we have a Lie algebroid representation.
\end{proposition}
We now change our perspective slightly and think of weighted $A$-connections in terms of morphisms of certain bi-graded bundles: this is closer to the original definitions of linear $A$-connections as given by Fernandes \cite{Fernandes:2002},  the slightly more general notion of $\rho$-connections on anchored bundles as given by Cantrijn \& B.~Langerock \cite{Cantrijn:2003} (also see Popescu \& Popescu \cite{Popescu:2001}), and the relative tangent spaces of Popescu \cite{Popescu:1992}.
\begin{proposition}\label{prop:weighted connection morphism}
A weighted $A$-connection on a graded bundle $F_{k}$ is equivalent to a morphism in the category of bi-graded (super) bundles 
$$h_\nabla :  \tau^* \Pi A \longrightarrow  \Pi \sT F_{k},$$
such that the following diagram is commutative
\begin{center}
\leavevmode
\begin{xy}
(0,20)*+{ \tau^*\Pi A}="a"; (30,20)*+{\Pi \sT  F_{k}}="b";
(0,0)*+{\Pi A}="c"; (30,0)*+{\Pi \sT M}="d";%
{\ar "a";"b"}?*!/_3mm/{h_\nabla};%
{\ar "a";"c"}?*!/^4mm/{\rmp_{\Pi A}};%
{\ar "b";"d"}?*!/_6mm/{(\sT \tau)^\Pi};%
{\ar "c";"d"}?*!/^3mm/{\rho^\Pi};%
\end{xy}
\end{center}
\end{proposition}
\begin{proof}
In local coordinates, the morphism $h_\nabla$ must be of the form
$$\big(x^{a}, y_{w}^{I}, \rmd x^{b} , \rmd y^{J}_{w} \big)\circ h_\nabla =  \big(x^{a},~ y^{I}_{w},~ \zx^{i}Q_{i}^{b}(x),~ \zx^{j}\Gamma_{j}^{J}[w](x,y)  \big).$$
 It is then a matter of checking the transformation rules to show that $\Gamma_{j}^{J}[w](x,y)$ are the Christoffel symbols of a weighted $A$-connection. Thus, given a morphism $h_\nabla$ we can construct a weighted connection $\nabla$. The other direction immediately follows as locally any weighted connection is specified by its Christoffel symbols.  
\end{proof}
\begin{example}[Affine connections]  Given an affine connection  on $M$ (so a linear connection on $\sT M$) we can construct the following commutative diagram
\begin{center}
\leavevmode
\begin{xy}
(0,20)*+{ \Pi \sT M \times_M \sT M}="a"; (40,20)*+{\Pi \sT \big( \sT M \big)}="b";
(0,0)*+{\Pi \sT M}="c"; (40,0)*+{\Pi \sT M}="d";%
{\ar "a";"b"}?*!/_3mm/{h_\nabla};%
{\ar "a";"c"}?*!/^6mm/{\rmp_{\Pi \sT M}};%
{\ar "b";"d"}?*!/_6mm/{(\sT \tau)^\Pi};%
{\ar "c";"d"}?*!/^3mm/{\Id_{\Pi \sT M}};%
\end{xy}
\end{center}
The morphism $h_\nabla : \Pi \sT M \times_M \sT M \longrightarrow  \Pi \sT \big( \sT M \big) $ is then described in local coordinates by
$$\big(x^a, \dot{x}^b , \rmd x^c , \rmd \dot{x}^d \big) \circ h_\nabla = \big ( x^a, \dot{x}^b , \rmd x^c ,  \rmd x^b  \dot{x}^a\Gamma_{ab}^d(x)  \big ),$$
where $\Gamma_{ab}^d$ are the Christoffel symbols of the specified affine connection.
\end{example}
\begin{example}[Linear connections]
\label{exm:LinearConnections}
If we consider a vector bundle, so a graded bundle of degree $1$, then we recover the notion of a linear $A$-connection  as given by Fernandes \cite{Fernandes:2002}. The commutative diagram here is 
\begin{center}
\leavevmode
\begin{xy}
(0,20)*+{ \Pi A \times_M E}="a"; (40,20)*+{\Pi \sT E}="b";
(0,0)*+{\Pi A}="c"; (40,0)*+{\Pi \sT M}="d";%
{\ar "a";"b"}?*!/_3mm/{h_\nabla};%
{\ar "a";"c"}?*!/^4mm/{\rmp_{\Pi A}};%
{\ar "b";"d"}?*!/_6mm/{(\sT \tau)^\Pi};%
{\ar "c";"d"}?*!/^3mm/{\rho^\Pi};%
\end{xy}
\end{center}
The morphism $h_\nabla : \Pi A \times_M E \longrightarrow \Pi \sT E$ is locally given by
$$\big ( x^a, y^\alpha, \rmd x^b, \rmd y^\beta  \big)\circ h_\nabla  = \big ( x^a, y^\alpha,  \zx^i Q_i^b(x) , \zx^i y^\alpha \Gamma_{\alpha i}^\beta(x)\big). $$
This should then be compared with \cite[Definition 0.3]{Fernandes:2002} which is precisely the above commutative diagram up to a shift in parity. The reader should also compare this with Vitagliano \cite[Example 6.]{Vitagliano:2016}.
\end{example}

\begin{example}[Homogeneous nonlinear $A$-connections]\label{exm:HomogeneousConnections}
If we consider a higher order tangent bundle $F_{k} := \sT^{k}M$ then a weighted connection is equivalent to a \emph{homogeneous nonlinear connection}, sometimes also referred to as a $k$-\emph{linear connection} (see for example  de Andr\'{e}s, de Le\'{o}n \& Rodrigues \cite{Andres:1989}, Kure\v{s} \cite{Kures:1989} and Morimoto \cite{Morimoto:1970}).  As a slight generalisation, we can consider an arbitrary Lie algebroid over $M$. For simplicity consider the second order tangent bundle $\sT^2M$ and the following commutative diagram
\begin{center}
\leavevmode
\begin{xy}
(0,20)*+{ \Pi A \times_M \sT^2 M}="a"; (40,20)*+{\Pi \sT  \big( \sT^2 M\big)}="b";
(0,0)*+{\Pi A}="c"; (40,0)*+{\Pi \sT M}="d";%
{\ar "a";"b"}?*!/_3mm/{h_\nabla};%
{\ar "a";"c"}?*!/^4mm/{\rmp_{\Pi A}};%
{\ar "b";"d"}?*!/_6mm/{(\sT \tau)^\Pi};%
{\ar "c";"d"}?*!/^3mm/{\rho^\Pi};%
\end{xy}
\end{center}
In homogeneous coordinates $\big( x^a , x^{b, (1)}, x^{c,(2)} \big)$, see Example \ref{exm:highertangent}, any weighted $A$-connection on  $\sT^2 M$ gives rise to the morphism $h_\nabla:\Pi A \times_M \sT^2 M \rightarrow \Pi \sT \big(  \sT^2M\big) $ which is locally given by
\begin{align*}
& \rmd x^a \circ h_\nabla = \zx^i Q_i^a(x),\\
& \rmd x^{b,(1)} \circ h_\nabla =  \zx^i x^{a,(1)}\overset{(1)}{\Gamma}\vphantom{\Gamma}_{ai}^b(x),\\
& \rmd x^{c,(2)} \circ h_\nabla = \zx^i x^{a,(2)}\overset{(2)}{\Gamma}\vphantom{\Gamma}_{ai}^b(x) + \frac{1}{2!} \zx^i x^{a,(1)}x^{b,(1)}\overset{(2)}{\Gamma}\vphantom{\Gamma}_{bai}^c(x).
\end{align*}
For the case of the $k^\textnormal{th}$ order tangent bundle, again in homogeneous coordinates, the $l^\textnormal{th}$ term $(1\leq l \leq k)$ can be written as
$$\rmd x^{a,(l)} \circ h_\nabla = \zx^i x^{b,(l)}\overset{(l)}{\Gamma}\vphantom{\Gamma}_{bi}^a(x) + \sum_{\substack{ m_1 + \cdots m_n = l    \\n>1}} \frac{1}{n!} \zx^i x^{b_1, (m_1)} \cdots x^{b_n, (m_n)}\overset{(l)}{\Gamma}\vphantom{\Gamma}_{b_n \cdots b_1i}^a(x) .$$ 
\end{example}
\noindent \textbf{Statement.} In light of Proposition \ref{prop:weighted connection morphism} and the  above examples, we  can view weighted $A$-connections as being adapted to the structure of the graded bundle in the sense that the morphism $h_\nabla : \tau^* \Pi A \rightarrow \Pi \sT F_k$ is not just a bundle morphism, but  is a morphism of bi-graded super bundles.

\subsection{Gauge transformations}
Before we can prove the existence of weighted $A$-connections for any graded bundle, we will need the notion of gauge equivalence of weighted $A$-connections.
\begin{definition}
Consider an arbitrary Lie algebroid $(\Pi A, \rmd_A)$ and a fixed graded bundle $(F_k, \rmh)$, both over the same smooth manifold $M$.  The \emph{gauge group} $\textnormal{Gau}(F_k, \Pi A)$, consists of automorphisms of the bi-graded bundle  $\tau^* \Pi A$ such that the following diagram is commutative:
\begin{center}
\leavevmode
\begin{xy}
(0,20)*+{ \tau^*\Pi A}="a"; (30,20)*+{\tau^*\Pi A}="b";
(0,0)*+{\Pi A}="c"; (30,0)*+{\Pi A}="d";%
{\ar "a";"b"}?*!/_3mm/{\varphi};%
{\ar "a";"c"}?*!/^4mm/{\rmp_{\Pi A}};%
{\ar "b";"d"}?*!/_6mm/{\rmp_{\Pi A}};%
{\ar "c";"d"}?*!/^3mm/{\Id_{\Pi A}};%
\end{xy}
\end{center}
\end{definition}
Note that this definition of the gauge group does not actually depend on any choice of  Lie algebroid $(\Pi A, \rmd_A)$, and is thus really associated with the graded bundle itself. In fact, from the definition, it is clear that the gauge group is isomorphic to the group of vertical automorphisms of $F_k$. However, it will be more practical to define the gauge group as we have.\par 
The gauge group $\textnormal{Gau}(F_k, \Pi A)$ has a natural left action on the space of weighted $A$-connections (fixing a Lie algebroid), which is defined as 
\begin{align}\label{eqn:GaugeTrans}
& \textnormal{Gau}(F_k, \Pi A)\times \mathcal{C}(F_k, \Pi A)  \longrightarrow \mathcal{C}(F_k, \Pi A) \\
 \nonumber & \big( \varphi , \nabla \big) \mapsto \nabla^\varphi := \varphi^* \circ \nabla \circ (\varphi^{-1})^*.
\end{align}
This action on the space of weighted $A$-connections we refer to as a  \emph{gauge transformation}.
\begin{definition}\label{def:GaugeEquiv}
 Any pair of weighted $A$-connections, $\nabla$ and $\nabla^\prime$, are said to be \emph{gauge equivalent} if there exists a $\varphi \in \textnormal{Gau}(F_k, \Pi A)$ such that 
 $$\nabla^\prime = \nabla^\varphi := \varphi^* \circ \nabla \circ (\varphi^{-1})^*.$$
\end{definition}
\subsection{The existence of weighted connections}
We now proceed to show that the set of weighted $A$-connections for a given graded bundle and a given suitable Lie algebroid is non-empty. 
\begin{lemma}\label{lem:split connections}
Given  any split graded bundle $F_k^{\:\textnormal{split}} := E_{1} \times_{M}E_{2}\times_{M} \cdots  \times_{M} E_{k} $ and any Lie algebroid $(\Pi A, \rmd_{A})$, both over the same base manifold $M$, the corresponding set of weighted $A$-connections is non-empty. 
\end{lemma}
\begin{proof}
From K\v{r}i\v{z}ka \cite[Lemma 2.]{Krizka:2008} we know that given any vector bundle and any Lie algebroid both over the same base manifold, the set of  linear $A$-connections on that vector bundle is non-empty. Thus every `building vector bundle' of a split graded bundle can be equipped with a linear $A$-connection for any Lie algebroid over the same base. It is a standard construction to build from linear $A$-connections a linear $A$-connection on  a Whitney sum of vector bundles -- locally this is essentially just taking the sum of the respective Christoffel symbols. Thus, any split graded bundle admits linear $A$-connections, which can be considered as weighted $A$-connections via Example (\ref{exm:LinearConnections}).
\end{proof}
\begin{remark}
K\v{r}i\v{z}ka's proof relies on the existence of a smooth partition of unity subordinate to a chosen vector bundle atlas: the proof is essentially a modification of the standard proof of the existence of a linear connection on a vector bundle using the trivial connection and then extending this via a partition of unity. As we work in the real and smooth  category we have such partitions of unity at our disposal. Alternatively, the existence of $A$-connections reduces to the existence of  linear $\sT M$-connections using the anchor map.
\end{remark}
We will denote some specified weighted $A$-connection on a given split graded bundle  as $\nabla^{\textnormal{split}}$. 
\begin{example}
Consider a split graded bundle of degree $2$, $F_2^\textnormal{split} := E_1 \times_M E_2$. Let us employ local coordinates
$$\big( \underbrace{x^a}_{0}, \: \underbrace{y^\alpha}_{1}, \: \underbrace{z^\mu}_{2} \big),$$
and of course, admissible changes of local coordinates are linear in each fibre coordinate.  Now suppose we are given linear $A$-connections on $E_1$ and $E_2$, which are locally specified by the Christoffel symbols $\Gamma_{\alpha i}^{\beta}(x)$ and $ \Gamma_{\mu i}^{\nu}(x)$, respectively. Then we define 
$$\nabla^{\textnormal{split}}  := \zx^{i}Q_{i}^{a}(x)\frac{\partial}{\partial x^{a}} + \frac{1}{2!}\zx^{i}\zx^{j}Q_{ji}^{k}(x)\frac{\partial}{\partial \zx^{k}} + \zx^{i} y^\alpha \Gamma_{\alpha i}^{\beta}(x) \frac{\partial}{\partial y^\beta} +  \zx^{i} z^\mu \Gamma_{\mu i}^{\nu}(x)\frac{\partial}{\partial z^\nu}, $$ 
which clearly gives a weighed $A$-connection on $F_2^\textnormal{split}$. The local expressions for higher degree split graded bundles are analogous.
\end{example}
\begin{theorem}\label{thm:existence connections}
The set of weighted $A$-connections $\mathcal{C}(F_k, \Pi A)$ is non-empty for any graded bundle $F_k$ and any Lie algebroid $(\Pi A, \rmd_{A})$ both over the same base manifold $M$.
\end{theorem}
\begin{proof}
From Lemma \ref{lem:split connections} we know that weighted $A$-connections always exist for split graded bundles. Furthermore, we know that any graded bundle is isomorphic, but not canonically to a split graded bundle and that the splitting  acts as the identity on $M$ (Theorem \ref{theorem:splitting}). Let us fix a graded bundle and a  Lie algebroid both over the same base. Let us denote a chosen splitting as $\phi : F_{k} \rightarrow F_k^{\:\textnormal{split}}$. \par 
Given a weighted $A$-connection on the split graded bundle $\nabla^{\textnormal{split}}$, we can use the splitting $\phi$ to construct a weighted $A$-connection on the `unsplit' graded bundle as
$$\nabla_\phi := \phi^* \circ \nabla^{\textnormal{split}} \circ (\phi^{-1})^*.$$
As weighted $A$-connections on split graded bundles always exist (Lemma \ref{lem:split connections}) we conclude that $\mathcal{C}(F_k, \Pi A)$ is non-empty.
\end{proof}
The construction of a weighted $A$-connection from a split weighted connection one depends explicitly on the splitting.  However, we have the following proposition.
\begin{proposition}\label{prop:gaugeequiv}
Let $\psi, \phi : F_k \longrightarrow F_k^{\textnormal{split}}$ be two distinct splittings of the graded bundle $F_k$. Then the weighted $A$-connections $\nabla_\phi$ and $\nabla_\psi$ are gauge equivalent.
\end{proposition}
\begin{proof}
If the two weighted connections are gauge equivalent (see Definition \ref{def:GaugeEquiv}) then we can write
$$\nabla_\phi^\varphi =  \varphi^* \circ \nabla_\psi \circ (\varphi^{-1})^*,$$
and all we need to do is specify $\varphi$. Directly from the definition of the two weighted $A$-connections, it is  immediately clear that 
$$\varphi = \phi^{-1} \circ \psi,$$
is the element of the gauge group that we are looking for. Thus, the two weighted $A$-connections are indeed gauge equivalent.
\end{proof}
Given that gauge transformations are invertible, and so given any weighted $A$-connection we can construct a split weighted $A$-connection we have the following.
\begin{corollary}\label{coro:gaugeequiv}
Fix a Lie algebroid $(\Pi A, \rmd_A)$ and a graded bundle $(F_k, \rmh)$. Modulo gauge equivalence, there is a one-to-one correspondence between weighted $A$-connections on $F_k$ and linear $A$-connections on $F_k^{\textnormal{split}}$.
\end{corollary}
\begin{remark}
The question of the moduli space of all gauge equivalent (flat) weighted $A$-connections on a graded bundle reduces to the study of the moduli space of all gauge equivalent (flat) linear $A$-connections on the corresponding split graded bundle. That is, one needs only examine linear $A$-connections on Whitney sums of vector bundles. This is not at all  surprising as  any topological information about a graded bundle is insensitive to the choice of atlas we employ. Thus, from a \emph{topological} standpoint, graded bundles are no more than vector bundles in disguise. However, the \emph{geometry} of graded bundles is rich.
\end{remark}
\begin{example}\label{exm:RiemannianManifold}
Let $(M, g)$ be a Riemannian manifold. Then we canonically have the Levi--Civita connection. The Christoffel symbols of which we write as $\Gamma^c_{ba}(x)$. The transformation laws for the Christoffel symbols we write in the convenient form
$$\left( \frac{\partial x^{a'}}{\partial x^a} \right)\left( \frac{\partial x^{b'}}{\partial x^b} \right)\Gamma_{b'a'}^{c'} = \Gamma_{ba}^c \left( \frac{\partial x^{c'}}{\partial x^c} \right) + \frac{\partial^2 x^{c'}}{\partial x^b \partial x^a}.$$
Now consider the second order tangent bundle  $\sT^2 M$ (see Example \ref{exm:highertangent}). As is well known, given any symmetric affine connection we can construct a splitting
$$\phi : \sT^2 M \longrightarrow \sT M \times_M \sT M,  $$
which in local coordinates we write as
$$ \big ( x^a, y^b, z^c\big) \circ \phi = \left (x^a , x^{b,(1)}, x^{c,(2)}  {-} \half x^{a,(1)}x^{b,(1)}\Gamma_{ba}^c (x) \right),$$
in hopefully clear notation.  The reader can quickly see that the transformation rules for the Christoffel symbols for the Levi--Civita connection are such that this splitting is well defined. \par 
Let us assume that we have been given a linear $A$-connection on $\sT M$. This induces the obvious linear $A$-connection on $\sT M \times_M \sT M$, which as a weighted $A$-connection we write as
$$\nabla^{\textnormal{split}}  = \zx^i Q_i^a(x)\frac{\partial}{\partial x^a} + \frac{1}{2} \zx^i \zx^j Q_{ij}^k(x)\frac{\partial}{\partial \zx^k} + \zx^i y^b \Gamma_{bi}^c(x) \frac{\partial}{\partial y^c}  + \zx^i z^b \Gamma_{bi}^c(x) \frac{\partial}{\partial z^c}.$$
We can now use the canonical splitting induced by the Riemannian structure on $M$ to define a weighted $A$-connection on $\sT^2 M$  (see the proof of Theorem \ref{thm:existence connections}). A direct calculation produces 
\begin{align*}
&\nabla  = \zx^i Q_i^a\frac{\partial}{\partial x^a} + \frac{1}{2} \zx^i \zx^j Q_{ij}^k\frac{\partial}{\partial \zx^k} + \zx^i x^{b,(1)} \Gamma_{bi}^c \: \frac{\partial \hfill}{\partial  x^{c,(1)} \hspace{-15pt}}  \\
&+ \zx^i \left(x^{b,(2)} \Gamma_{bi}^c  + \frac{1}{2}x^{a,(1)}x^{b,(1)} \left( Q_i^d \frac{\partial \Gamma_{ba}^c}{\partial x^d} {+}\Gamma_{bd}^c\Gamma_{ai} ^d {-} \Gamma_{ba}^d\Gamma_{di}^c   {+} \Gamma_{ad}^c\Gamma_{bi} ^d  \right) \right)  \frac{\partial \hfill}{\partial  x^{c,(2)} \hspace{-15pt}}  \hspace{12pt} \,.
\end{align*}
Due to Proposition \ref{prop:gaugeequiv}, any other Riemannian structure on $M$ will produce a gauge equivalent weighted $A$-connection  on $\sT^2 M$ via this construction.  In other words, given a linear $A$-connection on $\sT M$, one can always choose an  auxiliary metric on $M$ to construct  a weighted $A$-connection on $\sT^2 M$, and any different choice of a metric leads to gauge equivalent weighted $A$-connections.
\end{example} 
\begin{example}
As a specialisation of the previous example, we can consider the Lie algebroid to be the antitangent bundle and take the connection to be the Levi--Civita connection. In this case, we can canonically construct a weighted connection on $\sT^2 M$ as
\begin{align*}
 &\Gamma^a_b[1]\big(x,x^{(1)} \big) = x^{c,(1)}\Gamma_{cb}^a(x),\\
 &\Gamma^a_b[2]\big(x,x^{(1)}, x^{(2)} \big) = x^{c,(2)}\Gamma_{cb}^a(x) 
 + \frac{1}{2} x^{c,(1)} x^{d,(1)}\left( \frac{\partial \Gamma_{dc}^a}{\partial x^b}(x) + \Gamma_{db}^e \Gamma_{ec}^a(x) {-} \Gamma_{dc}^e\Gamma_{eb}^a(x) {+} \Gamma_{cb}^e \Gamma_{ed}^a(x)\right).
\end{align*} 
The curvature of this weighted connection can be deduced by using
$$\nabla^2 =  \phi^* \circ (\nabla^{split})^2 \circ (\phi^{-1})^*\,,$$
due to natural properties of the Lie bracket for $\phi$-related vector fields.  A quick calculation shows that
\begin{align*}
\nabla^2 &= - \frac{1}{2} \rmd x^a \rmd x^b x^{c,(1)}R^f_{cba}(x)\:  \frac{\partial \hfill}{\partial  x^{f,(1)} \hspace{-15pt}}  \hspace{12pt}\\
& -\frac{1}{2} \rmd x^a \rmd x^b \left( x^{c,(2)} R^f_{cba}(x) + \frac{1}{2}x^{c,(1)} x^{d,(1)} \left(\Gamma^f_{de}R^e_{cba}(x) +   \Gamma^f_{ce}R^e_{dba}(x) - 2 \Gamma^e_{dc}R^f_{eba}(x)      \right)\right) \frac{\partial \hfill}{\partial  x^{f,(2)} \hspace{-15pt}}  \hspace{12pt}\,,
\end{align*}
where $R^a_{bcd}(x)$ are the components of the Riemann curvature tensor associated with the Levi--Civita connection. Thus, we see that if the Riemannian manifold $(M,g)$ is flat, then the curvature of the canonical weighted connection  vanishes. The converse statement is also true:  this is quite clear from the local expression for the curvature.
\end{example}
\subsection{The category of weighted connections}
A weighted connection we understand as a system $(\Pi A \times_M F, \rml, \rmh,  \rmd_A, \nabla)$.  For brevity, it will be convenient to write only $(\Pi A \times_M F, \nabla)$ as the other structures are implied by the `ingredients' needed to define a weighted $A$-connection. 
\begin{definition}
A \emph{morphism of two weighted connections}  $(\Pi A \times_M F, \nabla)$ and  $(\Pi A^\prime \times_{M^\prime} F^\prime, \nabla^\prime)$ is a morphism of supermanifolds 
$$\Phi : \Pi A \times_M F  \longrightarrow \Pi A^\prime \times_{M^\prime} F^\prime,$$
that satisfies the following further conditions:
 \begin{enumerate}
       \item $\Phi \circ \rml_t = \rml^\prime_t \circ \Phi$,
       \item  $\Phi \circ \rmh_s = \rmh^\prime_s \circ \Phi$,
 \end{enumerate}
 for all  $(s,t) \in \R^2$;
\begin{enumerate}
\setcounter{enumi}{2}
       \item $\nabla \circ \Phi^* = \Phi^* \circ \nabla^\prime.$
\end{enumerate}
\end{definition}
Evidently, such morphisms are composable and so we obtain the category of \emph{weighted connections}.

\noindent \textbf{Observations.} It is clear that underlying a morphism of weighted connections is a morphism of Lie algebroids $\phi :\Pi A \rightarrow \Pi A^\prime$. The category of flat weighted connections is defined in an obvious way and forms a full subcategory of the category of weighted connections.

\subsection{Quasi-actions of Lie algebroids on graded bundles}\label{subsec:Action}
The notion of an \emph{infinitesimal action of a Lie algebroid on a fibred manifold} was first given by Kosmann-Schwarzbach \& Mackenzie \cite{Kosmann-Schwarzbach:2002}, (also see Higgins \&  Mackenzie \cite{Higgins:1990}). For clarity, we give a slightly modified definition suited to our purposes.
\begin{definition}[Adapted from \cite{Kosmann-Schwarzbach:2002}]\label{def:qactions}
Let $(\Pi A, \rmd_A)$ be a Lie algebroid and $(F, \rmh)$ be a graded bundle both over the same base manifold $M$. Then an \emph{infinitesimal quasi-action of $A$ on $F$} is an $\R$-linear map
$$a : \Sec(A) \rightarrow \Vect(F),$$
that satisfies the following conditions:
\begin{enumerate}
\item the map $a$ is weight zero in the sense that $a(u)$ is a weight zero vector field;
\item the vector field $a(u)$ is projectable and projects to $\rho_u$;
\item  the map $a$ is $C^\infty(M)$-linear in the sense that $a(fu) = (\tau^*f)a(u)$;
\end{enumerate}
for all $u \in \Sec(A)$ and $f \in C^\infty(M)$.
In addition, if we have
\begin{enumerate}
\setcounter{enumi}{3}
\item  $a([u,v]_A) = [a(u), a(v)]$,
\end{enumerate}
for all $u$ and $v \in \Sec(A)$, then we have an \emph{infinitesimal action of $A$ on $F$}.
\end{definition}
We now want to show the relation between weighted $A$-connections and infinitesimal quasi-actions of Lie algebroids on graded bundles. To do this we first need to construct a notion of a horizontal lift within the framework we have thus presented. 
\begin{definition}\label{def:Hlift}
Let $\nabla \in \Vect(\tau^* \Pi A)$ be a weighted $A$-connection on a graded bundle $F$. Then the \emph{generalised horizontal lift} with respect to $\nabla$ is the $\R$-linear map defined as
\begin{align*}
H_\nabla  : & \:  \Sec(A) \longrightarrow \Vect(F)\\
 & u \mapsto H_\nabla(u) := \big( \rmp_F \big)_* [\nabla, \iota_u]\,.
\end{align*}
\end{definition}  
Note that for all $u$ the vector field $H_\nabla(u)$ is weight zero. In homogeneous local coordinates, we have
$$H_\nabla(u) = u^iQ_i^a \frac{\partial}{\partial x^a}  \: + \: u^i \Gamma_i^I[w]\frac{\partial}{\partial y_w^I}.$$
\begin{lemma}\label{lem:ProjLin}
Let $\nabla \in \Vect(\tau^* \Pi A)$ be a weighted $A$-connection on a graded bundle $F$. Then the generalised horizontal lift has the following properties:
\begin{enumerate}
\item $H_\nabla(u)(\tau^* f) = \tau^*(\rho_u(f))$;
\item $H_\nabla(fu) = (\tau^* f) H_\nabla(u)$,
\end{enumerate}
for all $u \in \Sec(A)$ and $f\in C^\infty(M)$. 
\end{lemma} 
\begin{proof}
Both these statements are obvious from the global constructions and/or the local expressions.
\end{proof}
\noindent \textbf{Observations.} We can use either $\rmd_A$ or \emph{any} weighted $A$-connection $\nabla$ to define the Lie algebroid bracket (and the anchor) irrespective of if the weighted $A$-connection is flat or not.
\begin{lemma}\label{lem:liftmorph}
If $\nabla$ is a flat weighted $A$-connection, then the associated horizontal lift is Lie algebra morphism, i.e.,
$$H_\nabla\big( [u,v]_A \big) = [H_\nabla(u), H_\nabla(v)] .$$
\end{lemma}
\begin{proof}
Using the observation above that we can replace $\rmd_A$ with $\nabla$ in the definition of the Lie algebroid bracket, we have that
$$H_\nabla\big( [u,v]_A\big) = \big( \rmp_F \big)_*  [ [\nabla ,[\nabla, \iota_u]], \iota_v].$$
Then using the Jacobi identity twice we arrive at 
$$\big( \rmp_F \big)_*  [ [\nabla ,[\nabla, \iota_u]], \iota_v] =\big( \rmp_F \big)_*[ [\nabla, \iota_u], [\nabla , \iota_v] ] + \frac{1}{2} \big( \rmp_F \big)_* [[[\nabla, \nabla], \iota_u], \iota_v].$$
From Definition \ref{def:Hlift} and the properties of the pushforward we have that
$$H_\nabla\big( [u,v]_A \big) = [H_\nabla(u), H_\nabla(v)] + [ [\nabla^2 , \iota_u], \iota_v], $$
noting that we can drop the pushforward in the second term. Then assuming $\nabla^2 =0$ we obtain the desired result.
\end{proof}
\begin{theorem}\label{thm:actions}
There is a one-to-one correspondence between infinitesimal actions of $A$ on $F$ and flat weighted $A$-connections on $F$. 
\end{theorem}
\begin{proof}
Starting from a weighted $A$-connection Definition \ref{def:Hlift} together with Lemma \ref{lem:ProjLin} we can construct an infinitesimal quasi-action of $A$ on $F$. If the  weighted $A$-connection is flat, then Lemma \ref{lem:liftmorph} tells us that we have an infinitesimal action. \par
In the other direction, (1), (2), (3) of Definition \ref{def:qactions} tell us that   the local form of any infinitesimal quasi-action must be 
$$a(u) = u^iQ_i^a \frac{\partial}{\partial x^a}  \: + \: u^i \Gamma_i^I[w]\frac{\partial}{\partial y_w^I},$$
and  we can identify $\Gamma_i^I[w](x,y)$ as the Christoffel symbols of a weighted $A$-connection due to the invariance of $a(u)$ under changes of homogeneous coordinates, i.e, we have the transformation law \eqref{eqn:ChangeChristoffel}. Thus we construct a weighted $A$-connection from an infinitesimal quasi-action.  Definition \ref{def:Hlift} together with Lemma \ref{lem:liftmorph} show that if we start with an infinitesimal action then we construct in this way a flat weighted $A$-connection.
\end{proof}
The generalised horizontal lift allows us to think of infinitesimal (grading preserving) diffeomorphisms of $F$  that are generated by sections of $A$ via
\begin{align}\label{eqn:QuasiAction}
& x^{a} \mapsto x^{a} + \epsilon \: u^{i}(x)Q_{i}^{a}(x), && y^{I}_{w} \mapsto y^{I}_{w} + \epsilon\: u^{i}(x)\Gamma_{i}^{I}[w](x,y),
\end{align}
\noindent where $\epsilon$ is an infinitesimal parameter  of weight zero.
\begin{remark}
In general one may require external parameters and constants that carry non-zero weight (possibly negative weight). However, we will not need such parameters in this paper and so will not dwell on the subtleties of their employment. 
\end{remark}
\begin{remark}
Brahic \&  Zambon \cite{Brahic:2017} in a preprint discuss $L_\infty$-actions of Lie algebroids on $\N$-manifolds, i.e., supermanifolds with an additional non-negative weight such that the Grassmann parity and this weight coincide (mod $2$). Given a Lie algebroid $A$ and an $\N$-manifold $\mathcal{M}$, they show that there is a one-to-one correspondence between such actions and homological vector fields on $ A[1]\times_M \mathcal{M}$ that project to the homological vector field $\rmd_A$  (also see Mehta \& Zambon \cite{Mehta:2012}). This is consistent with the constructions put forward in this paper, however, we are working in slightly different categories. In particular, the graded bundles discussed here are all real smooth manifolds and not supermanifolds. There is little doubt that the main ideas of the current paper can be generalised to weighted $A$-connections on graded super bundles. 
\end{remark}
\begin{remark}
From Lemma \ref{lem:liftmorph} we can understand the curvature of a weighted $A$-connection \eqref{eqn:Curvature} in a more classical way as the obstruction to the horizontal lift being a Lie algebra morphism. We can then define a `curvature tensor' as
\begin{align*}
R_\nabla(u,v) &:= H_\nabla\big([u,v]_A  \big) {-} [H_\nabla(u), H_\nabla(v)] \\
& = {-}u^i v^i \left( Q_j^a \frac{\partial \Gamma^I_i[w]}{\partial x^a}  \: {-}\:   Q_i^a \frac{\partial \Gamma^I_j[w]}{\partial x^a}  \: {-}\: Q_{ij}^k \Gamma_k^I[w] \right .\\
 & + \left .  \Gamma_j^J[w_0] \frac{\partial \Gamma_i^I[w]}{\partial y^J_{w_0}} \: {-} \:   \Gamma_i^J[w_0] \frac{\partial \Gamma_j^I[w]}{\partial y^J_{w_0}} \right) \frac{\partial}{\partial y_w^I}\, .
\end{align*} 
\end{remark}
\begin{example}
Continuing Example \ref{exm:RiemannianManifold}, we can write a quasi-action of $A$ on $\sT^{2} M$, where $(M, g)$ is a Riemannian manifold as
\begin{align*}
& x^a \mapsto x^a + \epsilon\:  u^i(x)Q_i^a(x),\\
& x^{a,(1)} \mapsto x^{a,(1)}  + \epsilon \: u^i(x)x^{b,(1)}\Gamma_{bi}^a(x)\\
& x^{a,(2)} \mapsto x^{a,(2)} +\epsilon \: u^i(x) \left(x^{b,(2)} \Gamma_{bi}^a(x)  + \frac{1}{2}x^{b,(1)}x^{c,(1)} \Gamma_{bci}^a(x) \right),
\end{align*}
where $u = u^i(x)e_i$ is a (local) section of $A$, and 
$$\Gamma_{bci}^a(x) := \left( Q_i^d(x) \frac{\partial \Gamma_{bc}^a}{\partial x^d}(x) {+}\Gamma_{ce}^a\Gamma_{bi} ^e(x) {-} \Gamma_{cb}^e\Gamma_{ei}^c (x)  {+} \Gamma_{be}^a\Gamma_{ci} ^a (x) \right).$$
\end{example}
\begin{example}
Consider a Poisson manifold $(M, \mathcal{P})$. Here we consider the Poisson structure as an even quadratic function on the anticotangent bundle $\Pi \sT^*M$. In natural adapted coordinates $(x^a, x^*_b)$, we have $\mathcal{P} = \half \mathcal{P}^{ab}(x)x^*_b x^*_a$. As is well-known, the anticotangent bundle comes equipped with a canonical Schouten structure (odd Poisson), which is locally given by
$$\SN{X,Y} = (-1)^{\widetilde{X}+1} \frac{\partial X}{\partial x^*_a} \frac{\partial Y}{\partial x^a} {-}  \frac{\partial X}{\partial x^a} \frac{\partial Y}{\partial x^*_a}\,, $$
and that the Jacobi identity for the Poisson bracket is equivalent to $\SN{\mathcal{P}, \mathcal{P}}=0$. Thus, $\big( \Pi \sT^* M,  \rmd_\mathcal{P}\big)$ is a Lie algebroid, where $\rmd_\mathcal{P} := \SN{\mathcal{P}, -}$ is the Lichnerowicz--Poisson differential. The associated Lie algebroid bracket on one-forms, i.e., sections of $\sT^*M$, is known as the Koszul bracket. \par 
Now consider a graded bundle $F$ over a Poisson manifold $(M , \mathcal{P})$.  In natural homogeneous coordinates a \emph{weighted contravariant connection} is given by
$$\nabla = \mathcal{P}^{ab}(x) x^*_b \frac{\partial}{\partial x^a} {-} \frac{1}{2!} \frac{\mathcal{P}^{bc}(x)}{\partial x^a} x^*_c x^*_b \frac{\partial}{\partial x^*_a} {+} \Gamma^{Ia}[w](x,y) x^*_a\frac{\partial}{\partial y^I_w}\,.$$
We know, via Theorem \ref{thm:existence connections}, that such weighted contravariant connections can always be found.  Thus, we can always construct a quasi-action generated by one-forms on $M$ acting on $F$ given by
\begin{align*}
x^a & \mapsto  x^a + \epsilon \: \mathcal{P}^{ab}(x) \omega_b(x),\\
y_w^I & \mapsto  y^I_w + \epsilon \: \Gamma^{Ia}[w](x,y)\omega_a(x),
\end{align*} 
where $\omega = \rmd x^a \omega_a(x)$ is a one-form and $\epsilon$ is an infinitesimal parameter  carrying no weight.
\end{example} 

\noindent \textbf{Statement.} The  quasi-action associated with a weighted $A$-connection gives rise to   grading preserving (infinitesimal) diffeomorphisms of the underlying graded bundle (see \eqref{eqn:QuasiAction}).  In this sense, we can view a weighted connection to be adapted to the structure of the graded bundle. 

\section{Connections adapted to multi-graded bundles} \label{sec:MultiWeightedCon}
\subsection{Multi-weighted $A$-connections}
 Recall that a multi-graded bundle or $n$-fold graded bundle $\big( F, \rmh^1, \cdots , \rmh^n \big)$ with $n \in \mathbb{N}$ is defined as a manifold with a collection of homogeneity structures that pair-wise commute, i.e., $\rmh^i_t \circ \rmh^j_s = \rmh^j_s \circ \rmh^j_t$  for any $i$ and $j = 1,2,\cdots, n$ and all $t \in \R$ and $s \in \R$ (see \cite{Grabowski:2009,Grabowski:2012}).   Homogeneous local coordinates can always be employed with the multi-weight taking values in $\mathbb{N}^n$. Similarly to the case of a graded bundle, the admissible changes of homogeneous coordinates respect the multi-weight. In this way we obtain an array of polynomial bundle structures over the base manifold $M := \rmh^1_0 \rmh^2_0 \cdots \rmh^n_0(F)$. 
\begin{example}
All n-fold vector bundles, e.g. double vector bundles, are examples of multi-graded bundles. Here each of the weights that make up the multi-weight is restricted to be either zero or one. 
\end{example}
 \begin{example}
 Following Example \ref{exm:highertangent}, iterated applications of the higher tangent functor give canonical examples of multi-graded bundles. Specifically, $\sT^k \big( \sT^l (\sT^m M)\big)$ is a triple graded bundle with the natural triple of homogeneity structures of degrees $k$, $l$, $m$ inherited from each higher tangent bundle. A coordinate system $(x^a)$ on $M$ gives rise to adapted or homogeneous coordinates $x^{a,(\alpha, \beta, \gamma)}$ of tri-weight $(\alpha, \beta, \gamma) \in \mathbb{N}^3$, where $0 \leq \alpha \leq k$, $0 \leq \beta \leq l$ and $0 \leq \gamma \leq m$. Given that a higher tangent bundle naturally has the structure of a graded bundle, the reader can easily convince themselves admissible changes of coordinates respect assignment of tri-weight.
 \end{example} 
 \begin{example}
 Given a graded bundle $F$ of degree $k$, the $l$-th order tangent bundle, $\sT^l F$ is naturally a double graded bundle. One of the homogeneity structures is the canonical homogeneity structure associated with the higher tangent bundle, while the other is the lift of the homogeneity structure on $F$. Homogeneous coordinates on $F$ give rise, via the $(\alpha)$-lift, to adapted or homogeneous coordinates on  $\sT^l F$. The reader can quickly convince themselves that admissible changes of coordinates respect the bi-weight. 
 \end{example}
 The notion of a multi-weighted $A$-connection follows verbatim from the notion of a weighted $A$-connection (Definition \ref{def:weighted connection}).  Consider an $n$-fold graded bundle $\big( F, \rmh^1, \cdots , \rmh^n \big)$ and a Lie algebroid $\big( \Pi A, \rmd_A, \rml \big)$ both over the same base manifold $M$.  Then as before, we can consider $\tau^* \Pi A \simeq \Pi A, \times_M F$ which is an $n+1$-fold graded super bundle with the ordering of the homogeneity structures being $(\rmh^1, \cdots ,\rmh^n , \rml)$. 
 \begin{definition}\label{def:multiweighted connection}
With the above notation, a \emph{multi-weighted $A$-connection} on an $n$-fold graded bundle $F$ is an odd vector field $\nabla \in \Vect\big(\tau^*\Pi A \big )$ of  $n+1$-weight $(0,0, \cdots, 0,1)$ that projects to $\rmd_{A} \in \Vect(\Pi A)$. If the Lie algebroid is the tangent bundle, i.e, $A = \sT M$, then we simply speak of a \emph{multi-weighted connection}. If a multi-weighted $A$-connection is `homological', i.e., $\nabla^2=0$, then we have a \emph{flat multi-weighted $A$-connection}.
\end{definition}
All of the main statements of this paper push through to the case of multi-graded bundles with minimal effort by taking into account the additional weights - we will give some details for double vector bundles in the next subsection. In particular, we have the following proposition.
\begin{proposition}\label{prop:multiweighted connection morphism}
A multi-weighted $A$-connection on an $n$-fold graded bundle $F$ is equivalent to a morphism in the category of  $n+1$-fold graded super bundles 
$$h_\nabla :  \tau^* \Pi A \longrightarrow  \Pi \sT F,$$
such that the following diagram is commutative
\begin{center}
\leavevmode
\begin{xy}
(0,20)*+{ \tau^*\Pi A}="a"; (30,20)*+{\Pi \sT  F}="b";
(0,0)*+{\Pi A}="c"; (30,0)*+{\Pi \sT M}="d";%
{\ar "a";"b"}?*!/_3mm/{h_\nabla};%
{\ar "a";"c"}?*!/^4mm/{\rmp_{\Pi A}};%
{\ar "b";"d"}?*!/_6mm/{(\sT \tau)^\Pi};%
{\ar "c";"d"}?*!/^3mm/{\rho^\Pi};%
\end{xy}
\end{center}
\end{proposition}
\noindent \textbf{Statement.} In light of Proposition  \ref{prop:multiweighted connection morphism} we view multi-weighted $A$-connection  as being adapted to the structure on   $n$-fold graded bundles in the sense that the morphism   $h_\nabla :  \tau^* \Pi A \rightarrow  \Pi \sT F$ is not just a bundle morphism but a morphism of $n+1$-fold graded super bundles.
\subsection{Connections adapted to double vector bundles} 
 Rather than discuss full generalities, we concentrate on the example of bi-weighted $A$-connections on double vector bundles. As far as we are aware, a notion of connections adapted to the structure of double vector bundles has not appeared in the literature before. In part, we expect this is because the original definition of Pradines \cite{Pradines:1974} is far less `geometric' in nature and complicated as compared to the reformulation in terms of commuting regular homogeneity structures.\par 
Let $\big( D, \rmh^1, \rmh^2  \big )$ be a double vector bundle. Recall that here the homogeneity structures are regular (see \eqref{eqn:HomoReg}), and that they commute in the sense that 
$$\rmh^1_t \circ \rmh^2_s  =  \rmh^2_s \circ \rmh^1_t,$$
 for all  $(t,s) \in \R^2$. Following \cite{Grabowski:2009}, we know that any double vector bundle can be equipped with homogeneous local coordinates
 $$\big( \underbrace{x^a}_{(0,0)},\: \underbrace{y^\alpha}_{(0,1)},\: \underbrace{z^\mu}_{(1,0)} ,\:  \underbrace{w^l}_{(1,1)}  \big),$$
 where we have indicated the bi-weight. Admissible changes of homogeneous coordinates are of the form
 \begin{align*}
 & x^{a'} = x^{a'}(x), && y^{\alpha'} = y^\beta T_{\beta}^{\:\: \alpha'}(x),\\
 & z^{\mu'} = z^\nu T_\nu^{\:\: \mu'}(x), &&  w^{l'}= w^m T_m^{\:\: l'}(x) + z^\nu y^\alpha T_{\alpha \nu}^{\:\:\: \mu'}(x).
 \end{align*}
We define the `side vector bundles'  of $D$ as $\rmh^1_0(D) := E^1$ and $\rmh^2_0(D) := E^2$, which come with naturally induced homogeneous coordinates $(x^a, y^\alpha)$ and  $(x^a, z^\mu)$, respectively. The core vector bundle $C$ is defined as all the elements in $D$ that project to zero under both $\rmh^1_0$ and $\rmh^2_0$. The core comes with naturally induced homogeneous coordinates $(x^a, w^l)$. The base manifold  we define as $ M := \rmh_0^1(\rmh^2_0 D) = \rmh_0^2(\rmh^1_0 D)$. 
 \begin{example}
A \emph{split double vector bundle} is a double vector bundle of the form $D = E_{(0,1)} \times_M E_{(1,0)} \times_M E_{(1,1)}$, where each $E \rightarrow M$ is a vector bundle were the bi-weight of the fibre coordinates has been indicated.
\end{example}
 \begin{example}
  Well-known examples of double vector bundles include $\sT(\sT M)$, $\sT^*(\sT M) \cong \sT(\sT^* M)$, and $\sT^*(\sT^* M)$. The reader can quickly check that  admissible changes of homogeneous coordinates is of the required form. 
 \end{example}
 \begin{example}
 Similarly to the  previous example, if $\tau : E \rightarrow M$ is a vector bundle, then $\sT E$ and $\sT^*E$ are both double vector bundles. Again, the reader can quickly check the admissible changes of homogeneous coordinates.  
 \end{example}

 A bi-weighted $A$-connection $\nabla$ on $D$  is an odd vector field of tri-weight $(0,0,1)$ on $\tau^* \Pi A$ (see Definition \ref{def:multiweighted connection}). In homogeneous coordinates any  bi-weighted $A$-connection is of the following form:
$$\nabla  = \zx^{i}Q_{i}^{a}(x)\frac{\partial}{\partial x^{a}} + \frac{1}{2!}\zx^{i}\zx^{j}Q_{ji}^{k}(x)\frac{\partial}{\partial \zx^{k}} +  \zx^{i} y^\alpha \Gamma_{\alpha i}^\beta(x) \frac{\partial}{\partial y^\beta} + \zx^{i} z^\mu \Gamma_{\mu i}^\nu(x) \frac{\partial}{\partial z^\nu} + \zx^i \big( w^l \Gamma_{li}^m(x) + z^\mu y^\alpha \Gamma_{\alpha \mu i}^m(x)\big)\frac{\partial}{\partial w^m}\,. $$

 \noindent \textbf{Observations.} A bi-weighted $A$-connection $\nabla$ on $D$ is projectable to a bi-weighted $A$-connection on $E^1 \times_M E^2$, or equivalently, a pair of weighted $A$-connections  one on $E^1$ and the other on $E^2$. Moreover, we also have the underlying structure of a weighted $A$-connection on $C$. \par 
 In terms of a morphism of tri-graded super bundles $h_\nabla : \tau^* \Pi A \rightarrow \Pi \sT D$, we have
 \begin{align*}
 & \rmd x^a \circ h_\nabla = \zx^i Q_i^a(x),
 && \rmd y^\alpha \circ h_\nabla = \zx^i y^\beta \Gamma_{\beta i}^\alpha(x), \\
 & \rmd z^\mu \circ h_\nabla = \zx^i z^\nu \Gamma_{\nu i}^\mu(x),
 && \rmd w^l \circ h_\nabla = \zx^i \big( w^l \Gamma_{li}^m(x) + z^\mu y^\alpha \Gamma_{\alpha \mu i}^m(x)\big).
\end{align*}  
 In exactly the same way as for weighted $A$-connections we can construct a generalised horizontal lift (see Definition \ref{def:Hlift}) and an infinitesimal  quasi-action of the Lie algebroid on $D$ (see  \eqref{eqn:QuasiAction}). In homogeneous coordinates the quasi-action is given by
 \begin{align*}
  &  x^a \mapsto x^a  + \epsilon \:u^i(x) Q_i^a(x),
 && y^\alpha \mapsto  y^\alpha + \epsilon \:  u^i(x) y^\beta \Gamma_{\beta i}^\alpha(x), \\
 &  z^\mu \mapsto  z^\mu + \epsilon \: u^i(x) z^\nu \Gamma_{\nu i}^\mu(x),
 && w^l \mapsto   w^l   + \epsilon u^i(x) \: \big( w^l \Gamma_{li}^m(x) + z^\mu y^\alpha \Gamma_{\alpha \mu i}^m(x)\big),
\end{align*} 
for any section $u \in \Sec(A)$. Naturally, these infinitesimal diffeomorphisms respect the bi-grading, i.e., they preserve the structure of the double vector bundle $D$.\par 
Any double vector bundle $D$ is noncanonically isomorphic to a split double vector bundle of the form  $D^\textnormal{split} :=  E^1 \times_M E^2 \times_M C$. Then same arguments as in the proof of Theorem \ref{thm:existence connections}, together with Proposition \ref{prop:gaugeequiv} and Corollary \ref{coro:gaugeequiv} we arrive at the following theorem that we shall end on.
\begin{theorem}\label{thm:existence bi connections}
Let $\big(D, \rmh^1, \rmh^2\big)$ be a double vector bundle over $M$, i.e., $\rmh_0^1(\rmh^2_0 D)=:M$. Furthermore, let $\big( \Pi A,\rmd_A, \rml\big)$ be a Lie algebroid over $M$. Then the set of   bi-weighted $A$-connections on $D$ is non-empty. Moreover, modulo gauge equivalence, there is a one-to-one correspondence between bi-weighted $A$-connections on $D$ and linear $A$-connections on $ E^1 \times_M E^2 \times_M C$.
\end{theorem} 
\section*{Acknowledgements}
We thank  Janusz Grabowski for his comments on earlier drafts of this work. Furthermore, we cordially thank the anonymous referee for their valuable comments and suggestions that helped to improve the overall presentation of this paper.


\end{document}